\documentclass[11pt,reqno]{amsart}
     
\usepackage{xifthen,setspace,bm}

\usepackage[utf8]{inputenc}
\usepackage[T1]{fontenc} 
\usepackage{bm}
\usepackage{amsmath,amsthm,amsfonts,amssymb,mathrsfs,stmaryrd,bigints,graphicx}
\usepackage{dsfont}
\usepackage{mathtools}
\usepackage[usenames]{color}
\usepackage{thmtools, thm-restate}

\usepackage[colorlinks=true,linkcolor=blue]{hyperref}

\DeclareFontFamily{U}{matha}{\hyphenchar\font45}
\DeclareFontShape{U}{matha}{m}{n}{
      <5> <6> <7> <8> <9> <10> gen * matha
      <10.95> matha10 <12> <14.4> <17.28> <20.74> <24.88> matha12
      }{}
\DeclareSymbolFont{matha}{U}{matha}{m}{n}
\DeclareMathSymbol{\varleftrightarrow}{3}{matha}{"D8}
\DeclareMathSymbol{\nvarleftrightarrow}{3}{matha}{"DC}

\usepackage{epstopdf} 
\usepackage{amssymb}
\usepackage{setspace}
\usepackage{enumerate}
\usepackage{bigstrut}
\usepackage{esvect}
\usepackage{multirow}
\usepackage{mathtools,xparse}
\usepackage{mathrsfs}
\usepackage{hyperref}
\usepackage{xcolor}
\newtheorem*{lemma*}{Lemma}
\newtheorem{theorem}{Theorem}[section]

\newtheorem{lemma}[theorem]{Lemma}
\newtheorem{proposition}[theorem]{Proposition}

\theoremstyle{definition}
\newtheorem{definition}{Definition}
\newtheorem{remark}[theorem]{Remark}

\allowdisplaybreaks
\usepackage{chngcntr}

\renewcommand{\P}{\mathbb{P}}

\newcommand\ba{\boldsymbol{a}}

\newcommand\bx{\boldsymbol{x}}
\newcommand\by{\boldsymbol{y}}
\newcommand\bz{\boldsymbol{z}}
\newcommand\bw{\boldsymbol{w}}

\newcommand{\E}{{\mathbb{E}}}
\newcommand{\1}{\mathds{1}}

\newcommand{\Z}{\mathbb{Z}}

\newcommand{\R}{\mathbb{R}}

\newcommand{\e}{\varepsilon}

\newcommand{\N}{\mathbb{N}}

	\renewcommand{\P}{\mathbb{P}}

\newcommand{\cA}{\mathcal{A}}

\newcommand{\cH}{\mathcal{H}}
\newcommand{\cI}{\mathcal{I}}

\newcommand{\cK}{\mathcal{K}}

\newcommand{\cZ}{\mathcal{Z}}

\newcommand{\sumtwo}[2]{\sum_{\substack{#1 \\ #2}}}
\newcommand{\wh}{\widehat}

\newcommand\sfp{\mathsf p}

\newcommand{\f}{\frac}

\newcommand{\grad}{\nabla}

\renewcommand{\setminus}{\backslash}

\newcommand\dd{\mathrm{d}}
\def\ba{\begin{align}}
\def\ea{\end{align}}
\def\bs{\begin{split}}
\def\es{\end{split}}

\begin{document}

\title{Sharp moment and upper tail asymptotics for the critical 2d Stochastic Heat Flow}
\author{Shirshendu Ganguly,   Kyeongsik Nam}
 
\begin{abstract}
While $1+1$ dimensional growth models in the Kardar-Parisi-Zhang universality class have witnessed an explosion of activity 
over the last few decades, higher dimensional models remain much less explored. The special case of $2+1$ dimensions is particularly interesting as it is, in physics parlance, neither ultraviolet nor infrared super-renormalizable.   Canonical examples include the stochastic heat equation (SHE) with multiplicative noise and directed polymers. The models exhibit a weak to strong disorder transition as the inverse temperature, up to a logarithmic (in the system size) scaling, crosses a critical value $\beta_c$. While the sub-critical picture has been established in detail,  
 only very recently, in a breakthrough work, \cite{shfcritical} constructed a scaling limit of the critical $2+1$ dimensional directed polymer partition function. This was termed as the critical $2d$ Stochastic Heat Flow (SHF), a random measure on $\R^2.$ As is true with various naturally occurring random measures, the SHF is expected to exhibit rich intermittent behavior. A particular manifestation of this is the rather rapid growth of its moment sequence. Considering, for instance, the scalar random variable given by the mass the SHF assigns to the unit ball on the plane, its $h^{th}$ moment was known to grow at least as $\exp(C_1 h^{2})$ \cite{shfnotgmc} (a consequence of the Gaussian correlation inequality) and at most as $\exp(\exp (C_2h^2))$ \cite{shfuppermoment} for some positive constants $C_1$ and $C_2$.  The true growth rate, however, was predicted to be $\exp(\exp (C h))$ for some positive constant $C$ in the late nineties \cite{conjecture}.  
In this paper, we prove a lower bound of the $h^{th}$ moment which matches the predicted value, thereby exponentially improving the previous best lower bound. As a consequence we also obtain rather sharp bounds on its upper tail.  The key ingredient in the proof involves establishing a new connection of the SHF and moments thereof to the Gaussian Free Field (GFF) on related Feynman diagrams. This connection opens the door to the rich algebraic structure of the GFF to study the SHF. In particular, our proof makes use of Kirchhoff's Matrix-Tree theorem to reduce estimating
moments to counting spanning trees on Feynman diagrams. Along the way we also prove a new
monotonicity property of the correlation kernel for the SHF as a simple consequence of the domain
Markov property of the GFF.
\end{abstract}

\address{Department of Statistics, UC Berkeley, USA} 
\email{sganguly@berkeley.edu}

\address{Department of Mathematical Sciences, KAIST, South Korea}
\email{ksnam@kaist.ac.kr}

\maketitle

\tableofcontents

\section{Introduction}

Consider the two-dimensional Stochastic Heat Equation (SHE)
\begin{align}\label{SHE}
    \partial_t Z = \frac{1}{2}\Delta Z+\beta  \xi Z, \qquad t>0,   x\in \R^2,
\end{align}
where $\xi$ is space-time white noise.
As is well known, this is ill-posed on account of the roughness of the noise term. Nonetheless, through a series of rigorous works \cite{bertini1998two, shfsubcritical0, shfcriticalmoment, chatterjee, shfsubcritical, gu,  shfuppermoment, shfcritical}, a rather deep picture has emerged. 

The dimension $2$ case is particularly interesting since it is neither ultraviolet (like dimension $1$) nor infrared (like dimensions greater than $2$) super-renormalizable \cite{ultra}. Nonetheless, one can indeed make sense of a solution of \eqref{SHE}. This involves a regularization and a renormalization procedure. Regularization is usually done by mollifying, in space, the noise term with a smooth bump function on a small scale. It turns out that if the mollification scale is $\e,$ then the renormalization involves taking $\beta$ to be ${\hat \beta}/{\sqrt{\log  \left({1}/{\e}\right)}}$ where $\hat \beta$ is $O(1).$

The model undergoes a weak to strong disorder phase transition at $\hat \beta=\hat \beta_c$. {It turns out that $\displaystyle \hat \beta_c = \sqrt{\pi}$ for the discrete  approximation scheme of the directed polymer introduced shortly, whereas in the continuous approximation—obtained by a mollification of the SHE in \eqref{SHE}—the critical threshold is  $\displaystyle \hat \beta_c = \sqrt{2\pi}$  (see \cite[Appendix]{shfnotgmc} for the details of comparison of the critical windows of these approximations).} While the focus of this paper is the critical regime, before diving into the contributions of this paper, to put things in context let us review the developments so far briefly.

The first computations on the $2d$ SHE
in the critical window were carried out in \cite{bertini1998two} motivated by the works on the delta-Bose gas in \cite{albeverio}. The critical scaling was rediscovered in the investigation of the broader context of marginal relevance carried out in \cite{shfsubcritical0} who also discovered the phase transition which, it turns out, was missed by the authors of \cite{bertini1998two}.
Subsequently, across the papers \cite{chatterjee, shfsubcritical0, gu}, with a comprehensive treatment in \cite{shfsubcritical}, the subcritical model, i.e., when {$\hat \beta <\hat \beta_c,$} has been studied in great detail. The supercritical regime still remains essentially unexplored but the critical case has witnessed some significant developments including the construction of its scaling limit -- the $2d$ critical Stochastic Heat Flow (SHF). To facilitate the upcoming more formal discussions and build intuition it will be helpful to introduce the pre-limiting model of the directed polymer in random environment (DPRE) at this stage. 
Towards this, define:
\begin{align}\label{polymer456}
    Z_{M,N}^{\beta}(x,y) :=\E\Big[  \exp\big(\sum^{N-1}_{n=M+1}(\beta\omega(n,S_n)-\lambda(\beta))\Big) \1_{\{S_{N}=y\}}   |  S_M=x\Big],
\end{align}
where $(S_n)_{n\geq 0}$ denotes a two-dimensional %%%%%%
simple random walk, 
%%%%%
whose law %%%%%%
and expectation 
are denoted %%%%%
by $\P$ and %%%%%%
$ \E$ respectively, and $(\omega_{n,x})_{n\in \N, x\in \Z^2}$ is a family %%%%%
of i.i.d. random variables %%%%%%
with mean $0$ and variance $1$ %%%%%
having a finite log-moment generating function $\lambda(\beta):=\log\E\big[ e^{\beta\omega} \big]<\infty$ ($\forall \beta\in \R$), which serve as the %%%%%
discrete analogue of a space-time white %%%%%
noise (for the rest of the article, we can simply assume that the variables are i.i.d. standard Gaussians). In other words, $Z_{M,N}^{\beta}(x,y)$ is the %%%%%
partition function of %%%%%
the directed polymer model %%%%%
where each space-time lattice point $(t,x)$  is equipped %%%%%
with an i.i.d. variable $\omega(t,x).$

Before proceeding further, it is worth remarking, however, that while the DPRE has served as the key discretized model in much of the developments around the study of the critical $2+1$ dimensional stochastic heat equation, it is unclear at this stage what the scope of the universality of the  SHF and the associated, yet to be constructed, $2+1$ dimensional KPZ  equation is, unlike the one dimensional story. The latter is known to encompass a range of models beyond polymers, including random matrices, particle systems, models of interface growth, and models of random geometry such as first and last passage percolation \cite{kpz1, kpz2, kpz3, kpz4, kpz5}.

The Critical $2d$ SHF was %%%%%%
constructed in \cite{shfcritical}  as the unique %%%%%%
limit of the fields %%%%%
\begin{align}\label{preSHF}
    \cZ_{N;  s,t}^\beta(\dd x, \dd y):=
    \f{N}{4}
    Z _{[  Ns ],[ Nt ]}^{\beta_N}
 (  \llbracket  \sqrt{N}x    \rrbracket
    ,
    \llbracket   \sqrt{N}y   \rrbracket
) 
    \dd x \dd y, \qquad 0\le s<t<\infty  ,
\end{align}
where $[\cdot]$ maps a real %%%%%%
number to %%%%%%
its nearest even %%%%%%
integer neighbour, %%%%%%
$\llbracket\cdot\rrbracket$ maps $\R^2$ %%%%%%
points to %%%%%%%%%%
their nearest even %%%%%%
integer point on %%%%%%
$\Z^2_\text{even}:=\{ (z_1,z_2)\in\Z^2:z_1+z_2\in 2\Z \}$, %%%%%%
and $\dd x  \dd y$ denotes the Lebesgue %%%%%%
measure on $\R^2 \times \R^2$. %%%%%%
In this setting %%%%%%
one needs %%%%%%
to choose $\beta=\beta_N$ such that
\begin{align}\label{choiceb}
\sigma_N^2:= e^{\lambda(2\beta_N)-2\lambda(\beta_N)}-1 = \frac{1}{R_N} \Big(1+\frac{\theta+o(1)}{\log N} \Big),
\end{align}
{where $R_N$ denotes the expected number of collisions or overlap of two independent simple random walks up to time $N$. The asymptotics of the latter is given by $R_N = \frac{\log N}{\pi} + \bar c + o(1)$ for some absolute constant $\bar c$} and $o(1)$ denotes negligible %%%%%%
corrections as $N\to \infty$, and $\theta \in \R$ is a fixed constant (see \cite[Appendix A.1]{shfnotgmc} for details).  The main result is recorded next. See also the recent survey \cite{caravenna2024critical} for a comprehensive account of this and related developments.

\begin{theorem} 
    Let $\beta_N$ be as in \eqref{choiceb} for %%%%%%
some fixed $\theta \in \R$ and %%%%%%
$\big( \cZ_{N;  s,t}^{\beta_N}(\dd x, \dd y) \big)_{0\le s<t<\infty}$ 
    %%%%%%
be defined as in \eqref{preSHF}. %%%%%%
Then as $N\rightarrow\infty$, the process of %%%%%%
random measures %%%%%%
    $(\cZ_{N; s,t}^{\beta_N}(\dd x,\dd y))_{0\le s\le t<\infty}$ %%%%%%
converges in %%%%%%
finite dimensional %%%%%%
distributions to %%%%%%
a unique limit
    \begin{align*}
        \mathscr{Z}^\theta
        =
        (\mathscr{Z}_{s,t}^{\theta}(\dd x,\dd y))_{0\le s\le t<\infty},
    \end{align*}
    named the Critical %%%%%%
2d Stochastic Heat Flow.
\end{theorem}

Often, various field theories, which the SHF maybe viewed as one, exhibit rich fractal or intermittent behavior. A particularly canonical example of this is the Gaussian multiplicative chaos (GMC) obtained by exponentiating the Gaussian free field, see e.g.  \cite{gmc} for a beautiful survey on the developments around the construction of GMCs and the relevance to the study of quantum field theories. While, importantly, the SHF was proven to be not an exponential of a Gaussian process by proving certain moment inequalities in \cite{shfnotgmc}\footnote{The polymer measure associated to SHF was constructed in \cite{clark2024continuum} and may be viewed as a Gaussian multiplicative chaos on the path space in an appropriately conditional sense shown very recently in \cite{CT25}.}, it is still expected to exhibit rather intriguing intermittent behavior. The main motivation for this work is to initiate a program to carry out a comprehensive investigation of this.

Intermittent behavior in the $1+1$ dimensional polymer models and the associated solution to the stochastic heat equation (the logarithm of which is the Cole-Hopf solution of the KPZ equation) has witnessed impressive progress over the last decade. While a comprehensive review is beyond the scope of the paper, let us point the reader to some of the developments in {\cite{book,inter,inter2,ganguly,lil}}. Perhaps closer to the subject of the paper are the results on the intermittent properties of solutions of the Parabolic-Anderson model in two dimensions {\cite{fractal, fractal2, konig}}. Even for the SHF, certain results have recently been established in \cite{shfsingularity}. One of the main results in the latter focused on the SHF mass of a small ball of radius $\e,$ say $B=B(0,\e),$ i.e., the ball of radius $\e$ around the origin.  Namely, consider the random variable $X_\e$ given by 
\begin{align} 
X_\e:= \frac{\int_{x\in B}\mathscr{Z}_t^\theta(\dd x, \1)}{\pi \e^2}
\end{align}
where
$$\mathscr{Z}_t^\theta(\dd x, \1)
:=\int_{y\in \R^2} \mathscr{Z}_{0,t}^\theta(\dd x, \dd y).
$$

Note that $X_\e$ is the mass SHF assigns to $B$, i.e., the partition function of the polymer with one end point uniformly chosen inside $B$ and another completely free, relative to its Lebesgue measure. 
This normalization ensures that $\E(X_\e)$ has mean one for all $\e>0$. 

In \cite{shfsingularity} it was nonetheless shown that $X_\e \overset{p} \rightarrow 0$ ($\overset{p}\rightarrow$ denotes convergence in probability) as $\e \to 0.$
The approach was essentially to set up a comparison, via monotonicity of fractional moments, to a slightly subcritical model on a renormalized space-time lattice, say, corresponding to  ${\hat \beta}/{\sqrt{\log\left(\frac{1}{\e}\right)}} $ where $\hat \beta$ is strictly less than $\hat \beta_c$. It is known in this case, say from \cite{shfsubcritical}, that as $\e \to 0$, $X_\e \overset{p} \rightarrow W$ where $W$ is a log-Normal variable of mean one corresponding to a Gaussian with variance approximately  $\frac{1}{1-\left({\hat \beta}/{\hat \beta_c}\right)^2}$. It is straightforward to see that such a log-Normal variable must converge to zero in probability as $\hat \beta$ converges to $\hat\beta_c$. This along with the comparison indicated above suffices to prove the convergence to zero in probability for $X_\e$ in the critical case. In particular, this shows that the Critical $2d$ SHF is almost surely singular with respect to Lebesgue measure.

Nonetheless, such comparison results aren't powerful enough to get sharper results such as tail bounds which is the central motivation for the present work.  To precisely state the problem, take any smooth non-negative test function $\varphi$ on $\R^2$ such that $\varphi(0)>0$. Define 
\begin{align}\label{flatdata} 
\mathscr{Z}^\theta_{t}(\varphi):=\int_{\R^2}  \varphi(x)   \mathscr{Z}_t^\theta(\dd x, \1).
\end{align}
For simplicity, we will for the moment focus on the case $t=1.$ {A scaling covariance property of the SHF will allow us to translate the results from the $t=1$ case to any $t>0$ (see Remark \ref{scaling1} below).}
Further, for the choice of the test function, the case $\varphi=\mathbf{1}_B$ where $B=B(0,1)$ is the unit ball could be all the reader might want to keep in their mind. Let us also,  for notational brevity, denote  $\mathscr{Z}^\theta_{1}(\varphi)$ by 
\newcommand{\vp}{\varphi}
$X_{\vp}.$

Towards obtaining tail estimates we will mainly focus on the moment problem for $X_\vp.$ The problem has some history. In fact, the {tightness of the family $(\cZ_{N; 0,1}^{\beta_N}(\dd x,\dd y))_{N \in \mathbb N}$ follows immediately from a first-moment analysis, while the non-triviality of sub-sequential limits was first established by controlling the second and third moments of observables of the form $\mathscr{Z}^\theta_{1}(\varphi)$ in \cite{bertini1998two, shfcriticalmoment}.} Further, it was predicted in the late nineties in {\cite{conjecture}} that $\E(X^h_{\vp})$ should grow double exponentially in $h,$ i.e., {$\exp(\exp (ch))$ for some constant $c = c(\theta)>0$.} Nonetheless, mathematical progress towards this has still some ways to go.
In \cite{shfuppermoment}, using functional analytic tools the following upper bound was proven:
\begin{align}\label{upperbound1}
\E[X^h_{\vp}] \leq \exp(\exp(c h^2)), 
\end{align}
for some constant $c>0$ (see also \cite{chen1,chen2,chen3}).  On the other hand, the best lower bound until now, which is a consequence of the Gaussian correlation inequality \cite{gaussian1, gaussian2}, is
\begin{equation} \label{fkg12}
\exp\left({c}{h\choose 2}\right)
\end{equation}
{for some constant $c>0$,} 
and can be found, for instance, in \cite{shfnotgmc}.
One of the key reasons why the study of the SHF is delicate is because of the fast growth of its moments which renders the moment sequence indeterminate. In \cite{tsai}, it was shown however that the inbuilt independence in time and the convolution structure for the SHF allowed a Lindeberg type strategy to prove uniqueness provided the first few moments match.

\subsection{Results}
With the above context we are now in a position to state the results of this paper.
The main result of this article is an almost sharp lower bound, thereby obtaining an exponential improvement on the above stated previous lower bound.
\begin{theorem}\label{main} 
There exists an absolute constant $c_0>0$ such that the following holds.
{For any fixed $\theta \in \R$ and a smooth non-negative test function $\varphi$ on $\R^2$ such that $\varphi(0)>0$}, let $X_{\vp}=\mathscr{Z}^\theta_{1}(\varphi)$ be as above. Then for all large positive integer $h$,
\begin{align}  \label{mainalign}
 \E[ X^h_\vp]   \ge  \exp \big(\exp \left(c_0h \right)\big).
\end{align}
\end{theorem}

Along the way we also prove a useful monotonicity result about certain correlation kernels whose formal statement needs more preparation and is presented later as Proposition \ref{monotone}.

{
\begin{remark} \label{scaling1}
We remark  that for any $t>0$, the same bound \eqref{mainalign} holds for sufficiently large  $h$ (depending on $t$).  This follows from the following scaling property of the  Critical $2d$ SHF which is quoted from \cite[Theorem 1.2]{shfcritical}). For any $\mathsf{a}>0$,
\begin{align}  \label{scalingrelation}
(\mathscr{Z}^\theta_{\mathsf{a}s,\mathsf{a}t} (\dd (\sqrt{\mathsf{a}} x),\dd (\sqrt{\mathsf{a}} y)) )_{0\le s\le t < \infty}\overset{\text{law}}{=} (\mathsf{a} \mathscr{Z}^{\theta + \log \mathsf{a}}_{s,t} (\dd x,\dd y))_{0\le s\le t < \infty}.
 \end{align}
  We elaborate more on this in Remark \ref{scaling2} later.
\end{remark}
}

Finally, while a matching upper bound is still elusive (we will comment on what would be needed to be shown later in Remark \ref{upperarg}),  the current best upper bound of $\exp(\exp (ch^2))$ (with $c>0$) along with Theorem \ref{main} already allows us to obtain rather sharp  tail bounds of $X_\vp.$

{\begin{theorem} \label{tail}
  Let $\theta \in \R$ be fixed and $\varphi$  be a compactly supported smooth non-negative function on $\R^2$ such that $\varphi(0)>0$.  Then for all large enough $z,$
    \begin{align*}
      \exp\Bigl(-   (\log z)^{(\log\log z)^{1+o(1)}}\Bigr)\le   \P(X_\vp  > z)\le  \exp\Bigl(-\Omega(1) \cdot \log z  \cdot \sqrt{\log\log z} \Bigr).
    \end{align*}
    Here, $o(1) \rightarrow 0$ as $z\rightarrow \infty$ and $\Omega(1)$ remains bounded below by a positive constant.
    \end{theorem}}
    
    Note that the above shows that the tail is super-polynomial, albeit barely, i.e., ignoring $\log\log (z)$ terms, the tail  drops faster than any power.
On the other hand, for $d\geq 3$, in \cite{junk2024tail}, it was shown that the limiting partition function of the DPRE has a power law resulting in the finiteness of moments only up to a certain order.

\begin{remark}\label{optrem} A matching upper bound of {$\exp(\exp (ch))$} will allow us to upgrade the tail bound to be  
\begin{align*}
 \exp\Bigl(-   (\log z)^{{O(1)} \cdot \log \log \log z}\Bigr)\le   \P(X_\vp  > z)\le  \exp\Bigl(-\Omega(1) \cdot  \log z  \cdot \log\log z \Bigr),
    \end{align*}  
{where  $O(1)$ remains bounded above by a positive constant.}
   We elaborate more on this after the proof of the above theorem in Remark \ref{elab}. 
\end{remark}

\begin{remark} 
Sharp moment bounds in the microscopic case of $X_\e$ as $\e \to 0$ were computed previously in \cite{shrinking}.  
For all $h\geq 2$, $t>0$ and $\theta\in \R$ there exist a constant $C=C(h,\theta,t)$ such that
\begin{align}\label{thm:mom12}
C \big(\log \tfrac{1}{\e}\big)^{h \choose 2} 
\le \E\Big[ X^h_\e \Big] 
\le  \big(\log \tfrac{1}{\e}\big)^{{h \choose 2} +o(1) },
\end{align}
Note that the lower bound turns out to be sharp in this case. We will expand on this further later in the article in Remark \ref{erdostaylor}.
\end{remark}

\subsection{Acknowledgements} SG heartily thanks Francesco Caravenna, Rongfeng Sun and Nikos Zygouras for multiple long discussions, their willingness to explain various details of their important body of work on $2+1$ dimensional polymer models as well as their detailed comments that helped improve the article. In particular, a suggestion of Francesco Caravenna helped remove a lossy logarithmic factor in an earlier version of the main result.
SG was partially supported by NSF Career grant-1945172.
KN was supported by
Samsung Science and Technology Foundation under Project Number SSTF-BA2202-02. Part of
this work was completed during KN’s visit to UC Berkeley in the winter of 2024. The authors also thank Kaihao Jing for help with the figures.\\

In the upcoming Section \ref{sec2} we present a short overview of the key ideas that drive the proofs. We will also use this opportunity to  lay down some notational foundation and review the essential preliminaries for the  Critical  $2d$ SHF which will also allow us formally state the already alluded to monotonicity result.

\section{Preliminaries and key ideas} \label{sec2}
We start by recording some pertinent facts about the Critical $2d$ SHF (we refer to \cite{shfcritical} for the details). As is often the case, moments of polymer models are related to exponential moments of the number of collisions of independent random walks. A useful approach to study such
moments is via enumerating Feynman diagrams. This was indeed the method adopted in the work
of Carravena-Sun-Zygouras \cite{shfcritical,shfcriticalmoment}. Thus, some of the facts and formulas that appear next which might seem a bit mysterious stems from the analysis of Feynman diagrams carried out across the various papers \cite{shfcritical,shfcriticalmoment,shfuppermoment,shrinking}.

The first moment of the %%%%%%%
Critical $2d$ SHF is %%%%%%%
given by %%%%%
\begin{equation}\label{moment1}
	\E[\mathscr{Z}^\theta_{s,t}(\dd x, \dd y)]
	= \tfrac{1}{2}   g_{\frac{1}{2}(t-s)}(y-x)   \dd x   \dd y  ,
\end{equation}
where 
\begin{equation}\label{heatkernel1}
g_t(x):=\frac{1}{2\pi t}e^{-\frac{|x|^2}{2t}}
\end{equation}
denotes the %%%%%%%
two-dimensional heat kernel. %%%%%%%
The covariance kernel is as follows: %%%%%
\begin{equation}  
\begin{aligned}
	\text{Cov}[\mathscr{Z}^\theta_{s,t}(\dd x, \dd y), \mathscr{Z}^\theta_{s,t}(\dd x', \dd y')]
	&= \tfrac{1}{2}   K_{t-s}^\theta(x,x'; y, y')   \dd x   \dd y   \dd x'   \dd y'  ,
\end{aligned}
\end{equation}
where %%%%%
\begin{equation}
\label{eq:m2-lim}
\begin{split}
	K_{t}^{\theta}(x,x'; y,y')
	& :=  \pi \: g_{\frac{t}{4}}\big(\tfrac{y+y'}{2} - \tfrac{x+x'}{2}\big)
	\!\!\! \iint\limits_{0<a<b<t} \!\!\! g_a(x'-x)  
	G_\theta(b-a)   g_{t-b}(y'-y)   \dd a   \dd b  .
\end{split}
\end{equation}
In the above %%%%%%%
formula, $G_\theta(t)$ %%%%%%% 
denotes the %%%%%%%
density of the %%%%%%% 
renewal function of %%%%%%% 
 the {Dickman subordinator} constructed %%%%%%% 
 in \cite{shf2}, and its expression is given by %%%%%
\begin{align} 
G_\theta(t)=\int_0^\infty \frac{e^{(\theta-\gamma)s} st^{s-1}}{\Gamma(s+1)} \dd s,
\end{align}
%%%%%%%%%%%%%%%%%%%%
where $\gamma:=-\int_0^\infty ( \log u) e^{-u} \dd u $ denotes the Euler–Mascheroni  constant %%%%%%%
and $\Gamma(s)$ denotes the Gamma %%%%%%% 
function. The following %%%%%%% 
lemma provides the asymptotics %%%%%%% 
of $G_\theta(t)$ as $t \downarrow 0.$ %%%%%
\begin{lemma}[Proposition 1.6 in \cite{shf2}] \label{lemma asym} %%%%%%%%%%%%%%%%%%%%
As $t \downarrow 0$, 
%%%%%%%%%%%%%%%%%%%%%%%%%
\begin{equation} 
	G_\theta(t) = \frac{1}{t(\log\frac{1}{t})^2} \bigg\{ 1 + \frac{2\theta}{\log\frac{1}{t}}
	+ O\bigg(\frac{1}{(\log\frac{1}{t})^2}\bigg) \bigg\},
\end{equation}
and %%%%%%%%%%%%%%%%%%%%%%%%%
%%%%%
\begin{equation} 
\int_0^t G_\theta(s) {\rm d}s = \frac{1}{\log \frac{1}{t}} \bigg\{ 1 + \frac{\theta}{\log\frac{1}{t}}
	+ O\bigg(\frac{1}{(\log\frac{1}{t})^2}\bigg) \bigg\} .
\end{equation}
\end{lemma}
%%%%%
This is an apt point to mention that throughout the paper we will work with flat initial data.  
That is, in \eqref{flatdata} we chose the second coordinate in $\mathscr{Z}_t^\theta(\dd x, \dd y) := \mathscr{Z}_{0,t}^\theta(\dd x, \dd y) $ to be integrated over the entirety of $\R^2$ whereas the more general object would have been, for some test function $\psi,$
\begin{align}\label{flatdata1} 
\mathscr{Z}^\theta_{t}(\varphi, \psi):=\int_{\R^2\times \R^2}  \varphi(x)\psi(y) \mathscr{Z}_t^\theta(\dd x, \dd y).
\end{align}

Given the above setup, we are now in a position to present the key representation formula for the $h^{th}$-moment developed in \cite{shfuppermoment, shfcriticalmoment} (for the precise expression, the reader may also refer to the more recent \cite[Theorem 2.3]{shfnotgmc}). The expression, admittedly, is a bit unwieldy.   To help the reader gather better intuition, we will shortly explain how to interpret each term appearing in the expression.
Towards this, define  
\begin{align}\label{pairconv}
    \textsf{Pair}_h:= \{\{i,j\} : 1\le i<j\le h\},
\end{align}
i.e.  $|\textsf{Pair}_h| = {h \choose 2}.$ From now on, for every element $\{i,j\}$ in $    \textsf{Pair}_h$, we assume that $i<j.$
Then, 
\begin{align} \label{hformula}
2^h \cdot \E \Big[ &\big(\mathscr{Z}_t^\theta (\varphi)\big)^h\Big] =  \int_{(\R^2)^h}  \dd {\bz}  \varphi^{\otimes h}({\bz})    \Big\{ 1 +\!\!
       \sum_{m=1}^\infty
        (2\pi)^m
     \sumtwo{\{i_1,j_1\},...,\{i_m,j_m\} \in \textsf{Pair}_h}
    {\text{with $\{i_k,j_k\} \neq \{i_{k+1},j_{k+1}\}$ for $k=1,...,m-1$}}   \nonumber \\
    &        \iint\limits_{\substack{0\le a_1< b_1<...< a_m < b_m\le t \\ x_1,y_1,...,x_m, y_m \in  \R^2 }}     g_{\f{a_1}{2}} (x_1-z_{i_1})  g_{\f{a_1}{2}}  (x_1-z_{j_1})   
       \prod_{r=1}^m \Big[G_\theta (b_r-a_r)  g_{\frac{b_r-a_r}{4}} (y_r-x_r)\Big]   \nonumber \\
        & 	\cdot \prod_{r=1}^{m-1} \Big[g_{\frac{a_{r+1} - b_{\sfp(i_{r+1})}}{2}}(x_{r+1}-y_{\sfp(i_{r+1})}) g_{\frac{a_{r+1} - b_{\sfp(j_{r+1})}}{2}}(x_{r+1}-y_{\sfp(j_{r+1})})\Big]  \dd  \bx  \dd \by   \dd \boldsymbol{a} \dd \boldsymbol{b} \Big\},
\end{align} 
  where $$\varphi^{\otimes h}(\bz):=\varphi(z_1)\cdots \varphi(z_h),\qquad \bz = (z_1,\cdots,z_h) \in (\R^2)^h,$$
  and  for every pair $\{i_r,j_r\} \in \{1,\dots,h\}^2$ in the summation above, 
 \begin{align}\label{221}
     \sfp(i_r) :=  \max \{ 1\le k<r : i_r \in \{i_k,j_k\}\}, \qquad 2\le r\le m,
 \end{align} 
(set $\sfp(i_r):=0$ if the above set is empty). $\sfp(j_r)$ is similarly defined as well. Note that if $\sfp(i_r)=0$ (resp. $\sfp(j_r)=0$), then $(b_{\sfp(i_r)}, y_{\sfp(i_r)}) =(0,z_{i_r})$ (resp. $(b_{\sfp(j_r)}, y_{\sfp(j_r)}) =(0,z_{j_r})$). Also, for $r=1,$ we set $\sfp(i_{1})=\sfp(j_{1})=0$. 
Here $\sfp(\cdot)$ may be interpreted as denoting the parent in the natural directed structure that the diagram possesses (see Figure \ref{fig:feynmandiag} where the natural direction is forward where time increases).

\begin{figure}[hb]
    \centering
            \includegraphics[width= 4in]{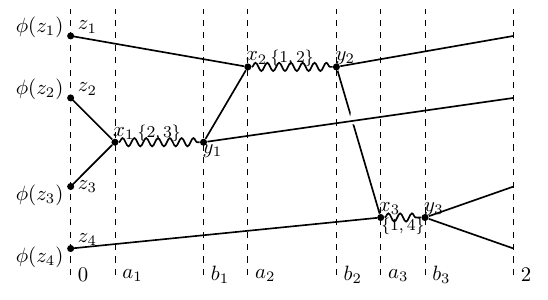}
    \caption{An illustration of a collision pattern in a Feynman diagram involved in the representation of the moment
formula \eqref{kernel} for $h=4$. In the diagram, the number of collisions is $m=3$ and there are $4$ Brownian motions starting from $z_1, z_2,z_3,z_4.$ The wiggle/curly lines
between points $(a_r, x_r)$ and $(b_r, y_r)$ are given weight $G_{\theta}(b_r-a_r)g_{\frac{b_r-a_r}{4}}(y_r-x_r)$ 
representing the total collision time of the tilted Brownian motion trajectories $B^{(i_r)}$ and $B^{(j_r)}$. Pairs $\{i_r, j_r\}$ above wiggle lines indicate the indices of the pair of Brownian motions involved in the collisions. 
Solid lines between points $(a_r, x_r)$ and $(b_{\mathsf{p}(i_r)}, y_{\mathsf{p}(i_r)})$ are weighted by the heat kernel $g_{\frac{a_r-b_{\mathsf{p}(i_r)}}{2}}(x_r-y_{\mathsf{p}(i_r)})$.
        }
        \label{fig:feynmandiag}
\end{figure}
~

A more compact expression involves clubbing together various terms in the above integrand into a kernel. To present this we first define the set of valid collision patterns.

\begin{definition}\label{colpat}
Let $h,m \in \mathbb N$. Define $\textsf{Col}^{(h,m)}$ to be the set of 
collision patterns $ \mathcal{I}= (\{i_1,j_1\},...,\{i_m,j_m\})$, where $\{i_1,j_1\},...,\{i_m,j_m\} \in \textsf{Pair}_h$ 
 with $\{i_k,j_k\} \neq \{i_{k+1},j_{k+1}\}$ for $k=1,...,m-1$.
\end{definition}
Note that
\begin{align}\label{col}
    |\textsf{Col}^{(h,m)}| = {h \choose 2} \Big[{h \choose 2} -1\Big]^{m-1}.
\end{align}

From the expression \eqref{hformula}, we obtain the following kernel integral formulation for the $h^{th}$ moment of the  Critical $2d$ SHF:
\begin{align} \label{hmoment}
    \E\Big[ \big( \mathscr{Z}_{t}^\theta(\varphi) \big)^h\Big]  =  \frac{1}{2^h} \int_{(\R^2)^h} \dd \bz    \varphi^{\otimes h}(\bz) 
 K_t^{(h)}(\bz) ,
\end{align}
where for $\bz = (z_1,\cdots,z_h)\in (\R^2)^h$, the kernel $ K_t^{(h)}(\bz) $ is defined as
\begin{align} \label{kernel}
  K_t^{(h)}(\bz) &:=  1 +\!\!
       \sum_{m=1}^\infty
        (2\pi)^m
     \sum_{(\{i_1,j_1\},...,\{i_m,j_m\}) \in \textsf{Col}^{(h,m)}}
    \nonumber \\
    &        \iint\limits_{\substack{0\le a_1< b_1<...< a_m < b_m\le t \\ x_1,y_1,...,x_m, y_m \in  \R^2 }}     g_{\f{a_1}{2}} (x_1-z_{i_1})  g_{\f{a_1}{2}}  (x_1-z_{j_1})   
       \prod_{r=1}^m \Big[G_\theta (b_r-a_r)  g_{\frac{b_r-a_r}{4}} (y_r-x_r)\Big]   \nonumber \\
        & 	\cdot \prod_{r=1}^{m-1} \Big[g_{\frac{a_{r+1} - b_{\sfp(i_{r+1})}}{2}}(x_{r+1}-y_{\sfp(i_{r+1})}) g_{\frac{a_{r+1} - b_{\sfp(j_{r+1})}}{2}}(x_{r+1}-y_{\sfp(j_{r+1})})\Big]  \dd  \bx  \dd \by   \dd \boldsymbol{a} \dd \boldsymbol{b} .
\end{align}

To interpret the above expression let us refer to Figure \ref{fig:feynmandiag}. As already alluded to, the moments of the SHF are given by Feynman diagrams encoding collision patterns of random walks whose initial data is given by the test function $\varphi.$ To see why this is the case, let us consider the polymer model and recall and examine the expression in  \eqref{polymer456} {with $(M,N) = (0,n)$}:
\begin{align}\label{polymer4567}
    Z_{0,n}^{\beta}(x,y) :=\E\Big[  \exp\big(\sum^{n-1}_{i=1}(\beta\omega(i,S_i)-\lambda(\beta))\Big) \1_{\{S_{n}=y\}}   |  S_0=x\Big].
\end{align}
The term $\lambda(\beta)=\frac{\beta^2}{2}$ appears to normalize the expectation (since the log-Laplace transform at $\beta$ of a standard Gaussian is $\frac{\beta^2}{2}$) leading to the first moment equality recorded in \eqref{moment1}.
Now the higher moments $\E\left[(Z_{0,n}^{\beta}(x,y))^h\right]$ will involve multiple random walks and them sharing the same noise. For instance, consider $h=2$, i.e., the second moment. Here two independent random walks intersect {(up to time $n$)} a random number of times, say, $\cI$. Each such collision gives rise to a term of the form 
$$\E[\exp\big(2\beta\omega(i,x)-2\lambda(\beta))]=\exp(\beta^{2})$$
where $(i,x)$ is the location of the collision.  Since each collision contributes the same, overall, one obtains, $\E(\exp(\beta^2 \cI))$. As is well known as the Erd\"os-Taylor theorem \cite{erdos}, $\cI$ is approximately distributed as an exponential variable {with parameter $\pi$} of scale $\log n$, leading to the choice of $\beta=\frac{\hat \beta}{\sqrt{\log n}}$ and $\hat \beta =\sqrt{\pi}$ being the critical location. 
For higher moments, the same principle holds and the contribution stems from various collision patterns of $h$ random walks as depicted in Figure \ref{fig:feynmandiag}. While, in principle, more than two random walks can collide simultaneously, it was already argued in \cite{shfcriticalmoment} that those represent a negligible fraction of the total moment and hence can be ignored in the scaling limit. This elucidates the pairwise interaction appearing in \eqref{kernel}.  
Thus we get 
\begin{equation}\label{pairwise}
\E\left[\exp \left(\beta^2 \sum_{1\le i < j \le h}\cI_{i,j}\right)\right],\end{equation}
where $\cI_{i,j}$ is the collision number of the random walks indexed $i$ and $j.$ An argument involving the Gaussian correlation inequality \cite{shfnotgmc} shows that the above is lower bounded by 
\begin{equation}\label{corrlb21}
\prod_{1\le i < j\le h }\E\left[\exp \left(\beta^2 \cI_{i,j}\right)\right].
\end{equation}
However, to get an exact expression in \eqref{kernel}, all possible number of pairwise interactions, $m$, is summed over. The collision times are $a_1<a_2<\ldots < a_m.$ The collision locations being $x_1, x_2 ,\ldots,x_m \in \R^2.$ The probability density of the collisions occurring at the above points account for the heat kernel term $g(\cdot)$ appearing in \eqref{kernel}. 
Finally, the source of the $G_{\theta}(\cdot)$ term is explained as follows.  Two random walks on the plane starting $\sqrt {n}$ apart and run for time $n$ are not very likely to meet. However it is a simple computation to check that the expected number of collisions is still $O(1).$ This is because the number of collisions have approximately the following heavy tailed distribution where it is $0$ with probability $1-\frac{1}{\log n}$ and $\log n$ with probability $\frac{1}{\log n}.$ Thus, on the event that the random walk collides, there are, unsurprisingly, many collisions in quick succession. The period of time witnessing the collisions is itself a random variable of scale $n$ which in \eqref{kernel} is given by the intervals $[a_r,b_r]$ with the first and last collision locations in this interval being $x_r, y_r$ respectively. The net contribution to the moment from this interval is $G_{\theta}(b_r-a_r).$ 

Finally, the $h$ independent random walk paths involved in the $h^{th}$ moment have their initial values essentially independently distributed according to density $\varphi$ which explains  the term $ \varphi^{\otimes h}(\bz).$  

\subsection{Key idea}\label{iop}
While the proof of Theorem \ref{main} is somewhat technical, let us nonetheless attempt to highlight the key idea. This may be summarized in one line as being able to capture the interactions between the spatial variables $x_{1},y_{1},\ldots , x_m,y_m\in \R^2.$ As already indicated, the previous best lower bound is a consequence of the Gaussian correlation inequality, and the upper bound arguments so far haven't fully dealt with the spatial aspect of the problem. The realization that this is unavoidable to obtain sharp bounds is the starting point of the paper. 

It turns out that the set of points $x_{1},y_{1},\ldots, x_m,y_m$ admits a rich spatial structure possessing key algebraic properties that we will take advantage of.  
We now briefly describe the latter. 
For any given Feynman collision diagram as in Figure \ref{fig:feynmandiag}, one may view it as a weighted graph $G=(V,E)$ with vertices $V=\{a_1,b_1, a_2,b_2,\ldots, a_m, b_m\}$ along with $h$ copies of $0$ acting as boundary vertices with the edge set being {$\{(a_r,b_r), (b_{\sfp(i_r)}, a_r),(b_{\sfp(j_r)}, a_r)\}_{1\le r\le m}$} with the notation $\sfp(\cdot)$ defined in \eqref{221}. Further, say the weight/conductance of an edge, $e,$  connecting $b_i$ and $a_j,$ is $c_e:=\frac{1}{|a_{j}-b_i|},$ i.e., the inverse of the distance between the points $a_j$ and $b_i$. The conductances for edges of the type $(a_i,b_i)$ need to be defined as $\frac{2}{|a_{i}-b_i|}.$ The spatial variables $x_{1}, y_{1},\ldots, x_m,y_m \in \R^2$ can now be viewed as the instance of a random function, say, $\Phi$ from $V\to \R^2$ such that for $1\le r \le m,$ we have $\Phi (a_r)=x_r$ and $\Phi(b_r)=y_r$ and the boundary condition at the copies of $0$ being given by $z_1,\ldots, z_h.$

At this point, fixing the weighted graph $G=(V,E)$, the dependence of the integrand in the expression \eqref{kernel} for the kernel $K_t^{(h)}$ on the function $\Phi$ can be seen to be 
\begin{equation}\label{gffenergy}
\exp\left(-\sum_{e\in E}{c_e\|\grad_{e}\Phi\|^2}\right),
\end{equation}
where $\grad_e$ denotes the discrete gradient along the edge $e$ and  $\|\cdot\|^2$ is the squared Euclidean norm.
The expert reader will notice that the above expression is exactly the density, up to constants, of a canonical Gaussian process, the Gaussian Free field (GFF) on the graph $G$ \cite{gff}.   
Given this, the kernel $K_t^{(h)}$ is obtained by integrating the above expression in \eqref{gffenergy} over the spatial variables, followed by `integrating' over all possible graphs, i.e., the vertex variables $\{a_1,b_1, a_2,b_2,\ldots, a_m, b_m\}$ as well as the choices of the edges $(a_{r}, b_{r}),(b_{\sfp(i_r)}, a_r),(b_{\sfp(j_r)}, a_r)$ for $1\le r\le m$. 
The expression in \eqref{gffenergy} allows us to use the rich algebraic structure that the GFF possesses. First, the integral of $\exp\left(-\sum_{e\in E}{c_e\|\grad_{e}\Phi\|^2}\right)$ over $\Phi$ is the partition function of the GFF, which is the square root of the determinant of Green's function of the random walk (one needs to enforce certain boundary conditions which we will not elaborate on in this discussion) (in fact our $\Phi$ is $\R^2$ valued and hence will correspond to two independent GFFs).  The latter is the inverse of the determinant of the corresponding weighted Laplacian. At this point a combinatorial fact comes to our aid: Kirchhoff's matrix-tree theorem. This says that the determinant of the Laplacian of a weighted graph is given by the count of weighted spanning trees on the same. This reduces the problem to estimating the latter which is how our argument proceeds.

Refraining from further discussions for the sake of avoiding technicalities, we end this section by mentioning that another important ingredient in our proof which is of independent interest is the following monotonicity result of the kernel $K_t^{(h)}(\bz)  $, defined in \eqref{kernel}, for $\bz = (z_1,z_2,\cdots,z_h) \in (\R^2)^h$. 

\begin{proposition}\label{monotone}
For any $(z_1,z_2,\cdots,z_h) \in (\R^2)^h$, the function $\alpha \mapsto K_t^{(h)}(\alpha z_1,\alpha z_2,\cdots,\alpha z_h)$ is non-increasing in $\alpha>0.$
\end{proposition}

The proof of the above proposition is another manifestation of the usefulness of the connection to GFF. The proof relies on the domain Markov property of the GFF, i.e., the distribution of $\Phi$ conditioned on the values on some vertices, say $v_{1}, \ldots v_k,$ interpreted as a boundary condition $\tau$, is given by the zero boundary GFF, i.e., where $\tau$ is set to $0$ plus the harmonic extension of $\tau$ on the remaining vertices.\\

Before proceeding further, let us outline the structure of the rest of the article next.  

\subsection{Organization of the article}
In Section \ref{sec3}, which contains a majority of the key ideas treats the case where the test function is the Gaussian heat kernel. Here the higher moment formula simplifies via the Chapman–Kolmogorov equation, and the connection to the  Gaussian Free Field is developed and used via the Matrix-Tree theorem. In Section \ref{sec4}, we prove the monotonicity result  Proposition \ref{monotone} as a consequence of the domain Markov property of the Gaussian Free Field. Section \ref{sec5} completes the proof of Theorem \ref{main} by extending our moment bounds from Gaussian to compactly supported test functions using a comparison result based on the monotonicity of the kernel. Finally, in Section \ref{sec6}, as a consequence of  Theorem \ref{main}, we derive sharp tail bounds, Theorem \ref{tail}, for the Critical  $2d$ SHF.

\section{Moments for Gaussian initial data} \label{sec3}
With the above preparation, we are now are in a position to dive into the proof of Theorem \ref{main}. However, as we will see, it will be convenient to initially start with a Gaussian test function instead of a generic $\varphi.$ This will be equivalent to considering the GFF on an augmented graph $\wh G$ obtained from $G$ by adding an extra vertex as illustrated in Figure \ref{fig:feynmandiag2}. 
Also for notational simplicity let us fix $t=1$ which will allow us to suppress the $t$ dependence from the notation. The same argument will continue to work for any general $t>0$ (see Remark \ref{scaling2} for the explanations).
Thus for the moment $\varphi=g_1$ where the latter is the heat kernel $g_1(x)=\frac{1}{2\pi}e^{-|x|^2/2},\ \forall x\in \R^2$ defined in \eqref{heatkernel1} and our object of interest is
 \begin{align} \label{heatkernel}
 \E\Big[ \big( \mathscr{Z}_{1}^\theta(g_1) \big)^h\Big]. 
 \end{align} 
 \begin{figure}[hb]
    \centering
            \includegraphics[width= 5in]{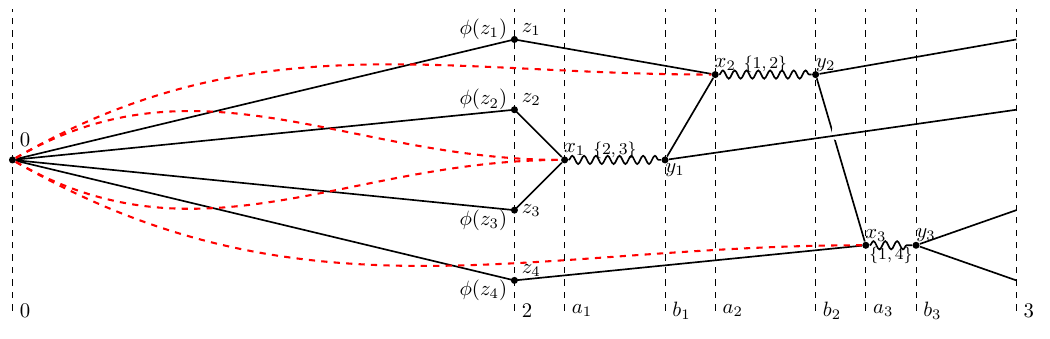}
    \caption{Illustration of the augmented graph $\wh G$ obtained from the graph $G$ in Figure \ref{fig:feynmandiag} by adding the auxiliary vertex $0$ and shifting the remaining coordinates by $2.$
 For our purposes, it will also be useful to consider the red edges which essentially replace the paths of length two passing through the points $z_1, \ldots, z_h$ by edges connecting their endpoints.         }
        \label{fig:feynmandiag2}
\end{figure}
The following is a version of Theorem \ref{main} in this case.
\begin{proposition}\label{heat}
There exists an absolute constant $c_1>0$ such that the following holds. For any  $\theta\in \R$, for sufficiently large $h$,
\begin{align} 
 \E\Big[ \big( \mathscr{Z}_{1}^\theta(g_1) \big)^h\Big]   \ge  \exp (e^{c_1h }).
\end{align}
\end{proposition}

The convenience of using $g_1$ as the test function now becomes apparent since using \eqref{hmoment} and the Chapman–Kolmogorov forward equation for the heat kernel, i.e., 
 \begin{align*}
  \int_{\R^2} g_{1}(z_\ell) \, g_{\frac{a_{r}}{2}}(x_{r}-z_\ell) \, \dd z_\ell 
  =g_{\frac{2+a_{r}}{2}}(x_{r}),
  \end{align*}
by shifting the time variables by $2$ the following equality holds.
\begin{align}
2^h \cdot \E\Big[ \big( &\mathscr{Z}_{1}^\theta(g_1) \big)^h\Big] =  1 +\!\!
       \sum_{m=1}^\infty
        (2\pi)^m
     \sum_{(\{i_1,j_1\},...,\{i_m,j_m\}) \in \textsf{Col}^{(h,m)}}
      \nonumber \\
    &        \iint\limits_{\substack{2\le a_1< b_1<...< a_m < b_m\le 3 \\ x_1,y_1,...,x_m, y_m \in  \R^2 }}     g_{\f{a_1}{2}} (x_1)^2 
       \prod_{r=1}^m \Big[G_\theta (b_r-a_r)  g_{\frac{b_r-a_r}{4}} (y_r-x_r)\Big]   \nonumber \\
        & 	\cdot \prod_{r=1}^{m-1} \Big[g_{\frac{a_{r+1} - b_{\sfp(i_{r+1})}}{2}}(x_{r+1}-y_{\sfp(i_{r+1})}) g_{\frac{a_{r+1} - b_{\sfp(j_{r+1})}}{2}}(x_{r+1}-y_{\sfp(j_{r+1})})\Big]  \dd \bx \dd \by   \dd \boldsymbol{a} \dd \boldsymbol{b}.
        \label{integrand}
\end{align} 
 We note that if $\sfp(i_r)=0$ (resp. $\sfp(j_r)=0$), then $(b_{\sfp(i_r)}, y_{\sfp(i_r)}) =(0,0)$ (resp. $(b_{\sfp(j_r)}, y_{\sfp(j_r)}) =(0,0)$).
Thus the main difference between \eqref{kernel} and the above is the term $g_{\f{a_1}{2}} (x_1)^2$ obtained as a consequence of the Chapman-Kolmogorov equation.

Given this formulation, we see below how given a positive integer $m$, a collision pattern $ \mathcal{I}= (\{i_1,j_1\},...,\{i_m,j_m\}) \in \textsf{Col}^{(h,m)}$, and times $2\le a_1< b_1<...< a_m < b_m\le 3$, i.e., the weighted graph $G=(V,E),$  the above integral with respect to the spatial variables  $x_1,y_1,...,x_m, y_m \in  \R^2$ is written as the partition function of the GFF with respect to an augmented graph $\wh G.$ 
The latter has an extra vertex $0$ connected to the vertices corresponding to the variables $z_1,\ldots, z_h$ (note however that the effect of applying the Kolmogorov-Chapman equation is to simply remove the latter vertices and consider the red edges bypassing them as shown in Figure \ref{fig:feynmandiag2}). 
 
Recalling that $g_t(x):=\frac{1}{2\pi t}e^{-\frac{|x|^2}{2t}}$, the spatial part of the integral in \eqref{integrand} has the following form.

\begin{align} \label{2}
  &  \iint\limits_{\substack{x_1,y_1,...,x_m, y_m \in  \R^2 }}    g_{\frac{a_1}{2} } (x_1)^2  \prod_{r=1}^{m}   \Big[  G_\theta (b_r-a_r) g_{\frac{b_r-a_r}{4}} (y_r-x_r)\Big]  \nonumber \\
  &\qquad \qquad \cdot \prod_{r=1}^{m-1}  \Big[g_{\frac{a_{r+1} - b_{\sfp(i_{r+1})}}{2}}(x_{r+1}-y_{\sfp(i_{r+1})}) g_{\frac{a_{r+1} - b_{\sfp(j_{r+1})}}{2}}(x_{r+1}-y_{\sfp(j_{r+1})})\Big]\dd \bx \dd \by  \nonumber \\
   &= \Big (  \frac{1}{\pi a_1} \Big)^2 \prod_{r=1}^{m}     G_\theta (b_r-a_r) \cdot \prod_{r=1}^{m} \frac{1}{2\pi \cdot  \frac{b_r-a_r}{4}}  \cdot \prod_{r=1}^{m-1} \Big[\frac{1}{ \pi   (  {a_{r+1} - b_{\sfp(i_{r+1})}} )  }   \cdot  \frac{1}{\pi   ({a_{r+1} - b_{\sfp(j_{r+1})}}) } \Big]\nonumber \\
   &\qquad \cdot 
\iint\limits_{x_1,y_1,\cdots,x_m,y_m\in \R^2} \exp \Big ( -\frac{|x_1|^2}{a_1}-\frac{|x_1|^2}{a_1}  - \sum_{r=1}^m \frac{|y_r-x_r|^2 }{(b_r-a_r)/2}  \nonumber \\
   &\qquad \qquad\qquad\qquad -\sum_{r=1}^{m-1} \frac{|x_{r+1}-y_{\sfp(i_{r+1})}| ^2}{  a_{r+1} - b_{\sfp(i_{r+1})}} -  \sum_{r=1}^{m-1}\frac{|x_{r+1}-y_{\sfp(j_{r+1})}| ^2}{  a_{r+1} - b_{\sfp(j_{r+1})}}\Big)\dd \bx \dd \by .
\end{align}
The integrand now is of the kind in \eqref{gffenergy}.

\subsection{GFF, weighted Laplacians and the Matrix-Tree theorem}
We now arrive at the already alluded to expression for the partition function of the GFF in terms of weighted spanning trees.

\begin{lemma}\label{key}
 Consider a finite  (weighted)  connected graph \( H=(V,E,c)\) with an arbitrary designated pinning vertex \(o\in V\). For any (undirected)  edge $\{u,v\}\in E$, let $c_{uv}=c_{vu}>0$ be its conductance. Then, setting $\phi_o:=(0,0) \in \R^2$, we have 
    \begin{align} \label{lemmaeq}
       \int_{(\R^2)^{V \setminus \{o\} }   } \exp\Big(-\frac{1}{2}\sum_{\{u,v\}\in E} c_{uv}|\phi_u - \phi_v|^2\Big)  \prod_{v\in V \setminus \{o\} } \dd \phi_v = (2\pi)^{|V|-1}  \Big (\sum_{T \in \mathcal T(H)} c_T \Big )^{-1},
    \end{align}
    where  \(\mathcal{T}(H)\) denotes the set of all spanning trees of \(H\) and $c_T := \prod_{e\in E(T)} c_e.$
\end{lemma}
The LHS \eqref{lemmaeq} is the partition function of the Gaussian Free Field on the weighted graph $H=(V,E,c)$ where $c$ is the conductance vector and where the GFF is pinned at $o$ to be $(0,0).$

~

Lemma \ref{key} is obtained as a consequence of the Gaussian integration along with the celebrated Kirchhoff's Matrix-Tree Theorem (see \cite{matrixtree}) which we state next.
For a finite connected (weighted) graph \(H=(V,E)\), the diagonal matrix \(D = (D_{uu})_{u\in V}\) is defined as:
\[
D_{uu} := \sum_{\substack{v\in V \\ \{u,v\} \in E}} c_{uv}, \qquad \forall u \in V,
\]
meaning that each diagonal entry \(D_{uu}\) is the sum of the conductances of all edges incident to vertex \(u\in V\). Next,
the adjacency matrix \(A=(A_{uv} )_{u,v\in V}\) is defined by
\[
A_{uv} := 
\begin{cases} 
c_{uv}, & \text{if } \{u,v\} \in E, \\
0, & \text{otherwise.}
\end{cases}
\]
The (un-normalized) Laplacian matrix is then defined as 
\begin{equation}\label{laplacian}
L :=D-A.
\end{equation}

\begin{lemma}\label{matrixtree}
    Let $L$  be the Laplacian matrix of the (weighted) {connected} graph $H = (V,E, c)$.  Then, for any fixed vertex \(o \in V\), the determinant of the \((|V|-1)\times (|V|-1)\) matrix \(L^{(o)}\) obtained by deleting the row and column of \(L\) corresponding to $o$ is equal to the number of (weighted) spanning trees   of \(H\), i.e.,
\[
\det(L^{(o)}) = \sum_{T \in \mathcal{T}(H)} c_T,
\]
where \(\mathcal{T}(H)\) denotes the set of all spanning trees of \(H\) and $c_T := \prod_{e\in E(T)} c_e.$
\end{lemma}

Given the above input, we now prove Lemma \ref{key}.
\begin{proof}[Proof of Lemma \ref{key}]
   As every $\phi_v$ takes values in $\R^2, $  we write $\phi_v = (\phi_{v,\mathsf{x}},\phi_{v,\mathsf{y}}) \in \R^2.$ Let $$\Phi_\mathsf{x} : = (\phi_{v,\mathsf{x}})_{v\in V} \in \R^{V} \text{ and } \Phi_\mathsf{y} : = (\phi_{v,\mathsf{y}})_{v\in V} \in \R^{V}.$$
Observe that the exponent in the Gaussian integral in the LHS in \eqref{lemmaeq} is the quadratic form induced by $L$, i.e.,
\begin{align} \label{213}
   \frac{1}{2}  \sum_{\{u,v\}\in E} \ c_{uv} |\phi_u - \phi_v|^2 & =  \frac{1}{2} \sum_{\{u,v\}\in E} \ c_{uv} (\phi_{u,\mathsf{x}} - \phi_{v,\mathsf{x}} )^2 + \frac{1}{2} \sum_{\{u,v\}\in E} \ c_{uv} (\phi_{u,\mathsf{y}} - \phi_{v,\mathsf{y}} )^2  \nonumber \\
   &= \frac{1}{2} \Phi_\mathsf{x}^T L \Phi_\mathsf{x} + \frac{1}{2} \Phi_\mathsf{y}^T L \Phi_\mathsf{y}  .
\end{align}
Recalling that the field is pinned at vertex \(o\) (i.e.,  \(\phi_o=(0,0)\)), the integration is carried out over the remaining vertices \(V\setminus\{o\}\), and thus the corresponding Laplacian is the reduced Laplacian \(L^{(o)}\) (i.e., obtained by deleting the row and column corresponding to \(o\)). Hence defining
$$\widetilde \Phi_\mathsf{x} : = (\phi_{v,\mathsf{x}})_{v\in V \setminus \{o\} } \in \R^{V\setminus \{o\} } \text{ and }\widetilde \Phi_\mathsf{y} : = (\phi_{v,\mathsf{y}})_{v\in V\setminus \{o\} } \in \R^{V\setminus \{o\} },$$
the quantity \eqref{213} is written as
\begin{align*}
    \frac{1}{2} \widetilde\Phi_\mathsf{x}^T L ^{(o)}\widetilde\Phi_\mathsf{x} + \frac{1}{2} \widetilde\Phi_\mathsf{y}^T L ^{(o)}\widetilde\Phi_\mathsf{y} .
\end{align*}
Now a standard fact about Gaussian integrals yields
\begin{align*}
\int_{\mathbb{R}^{V\setminus\{o\}}} \exp\left(-\frac{1}{2} \widetilde\Phi_\mathsf{x}^T L^{(o)}\widetilde\Phi_\mathsf{x} \right) \prod_{v\in V \setminus \{o\} } \dd \phi_{v,\mathsf{x}}& = \int_{\mathbb{R}^{V\setminus\{o\}}} \exp\left(-\frac{1}{2} \widetilde\Phi_\mathsf{y}^T L^{(o)}\widetilde\Phi_\mathsf{y} \right) \prod_{v\in V \setminus \{o\} } \dd \phi_{v,\mathsf{y}} \\
&= (2\pi)^{\frac{|V|-1}{2}}  (\det L^{(o)})^{-1/2}.
\end{align*}
{Note that \(L^{(o)}\) is positive definite, so the above Gaussian integral is well-defined. Indeed, for any vector \(\mathbf{w}=(w_v)_{v\in V\setminus\{o\}}  \in \R^{V\setminus \{o\}}  \) ,
\[
\mathbf{w}^T L^{(o)} \mathbf{w}
=\sum_{\{u,v\}\in E}c_{uv}\,(w_u - w_v)^2
\]
(the ``pinned'' coordinate \(w_o\) is fixed to zero).
Since \(H\) is connected and $c_{uv}=c_{vu}>0$, this quadratic form vanishes only when all \(w_v\), including $w_o$, are equal. As \(w_o=0\), the only admissible solution is \(\mathbf{w}=\textbf{0}\), proving positive definiteness.}
Thus the LHS of \eqref{lemmaeq} is equal to
\begin{align*}
    (2\pi)^{|V|-1}  (\det L^{(o)})^{-1}.
\end{align*}
 Lemma \ref{matrixtree} now finishes the proof.

\end{proof}

\subsection{Counting weighted spanning trees} To prove Proposition \ref{heat}, with the aid of Lemma \ref{key},
we aim to lower bound the quantity in \eqref{2}.  To apply the former, we have to first fix a realization of the weighted graph $G,$ which amounts to fixing the time slices $2\le a_1< b_1<...< a_m < b_m\le 3$ as well as the collision pattern  $ \mathcal{I}= (\{i_1,j_1\},...,\{i_m,j_m\}) \in \textsf{Col}^{(h,m)}$ (see Figure \ref{fig:feynmandiag2}).\\

We will restrict to the case where the time slices $2\le a_1< b_1<...< a_m < b_m\le 3$ satisfy for every $i=1,2,\cdots,m,$
\begin{align} \label{a}
    a_i \in \Big[2+ \frac{2i-1}{2m}  - \frac{1}{10m}, 2+ \frac{2i-1}{2m}  + \frac{1}{10m}\Big],
\end{align} 
and,
\begin{align} \label{b}
    b_i \in \Big[a_i, 2+ \frac{2i}{2m}  - \frac{1}{5m}\Big].
\end{align} 
These conditions ensure that for any collision pattern  $ \mathcal{I}= (\{i_1,j_1\},...,\{i_m,j_m\}) \in \textsf{Col}^{(h,m)}$ and $1\le r\le m-1$,
{\begin{align}\label{dominance}
{  0\le b_r-a_r  \le \min_{1\le r'\le m-1} \{ a_{r'+1} - b_{\sfp(i_{r'+1})},a_{r'+1} - b_{\sfp(j_{r'+1})}\}}
\end{align}}
and
{\begin{align}\label{dominance2}
     0\le  b_r-a_r \le \frac{2}{5m} \le  \frac{1}{5}a_1.
\end{align}}
Note that \eqref{dominance} in particular implies that all the `curly' edges in the graph, as in Figure \ref{fig:feynmandiag2}, have lengths less than all the `non-curly' edges. 
We next, under the conditions \eqref{a} and \eqref{b}, lower bound the integral appearing in \eqref{2}, i.e.,
\begin{align} \label{215}
&\iint\limits_{x_1,y_1,\cdots,x_m,y_m\in \R^2} \exp \Big (-\frac{ |x_1|^2}{a_1}   -\frac{ |x_1|^2}{a_1}  - \sum_{r=1}^m \frac{|y_r-x_r|^2 }{(b_r-a_r)/2}  \nonumber \\
   &\qquad \qquad\qquad\qquad -\sum_{r=1}^{m-1} \frac{|x_{r+1}-y_{\sfp(i_{r+1})}| ^2}{  a_{r+1} - b_{\sfp(i_{r+1})}} -  \sum_{r=1}^{m-1}\frac{|x_{r+1}-y_{\sfp(j_{r+1})}| ^2}{  a_{r+1} - b_{\sfp(j_{r+1})}}  \Big) \dd \bx \dd \by .
\end{align} 
Recalling the definitions from \eqref{221}, the Gaussian integral \eqref{215} can be written as
\begin{align} \label{gauss}
\iint\limits_{x_1,y_1,\cdots,x_m,y_m\in \R^2} \exp \Big (-\sum_{r=1}^m \frac{|y_r-x_r|^2 }{(b_r-a_r)/2} -\sum_{r=1}^{m} \frac{|x_r-y_{\sfp(i_r)}| ^2}{  a_r - b_{\sfp(i_r)}} -  \sum_{r=1}^{m}\frac{|x_r-y_{\sfp(j_r)}| ^2}{  a_r - b_{\sfp(j_r)}}  \Big) \dd \bx \dd \by .
\end{align}
We define a partition of the time index set $\{1,2,\dots,m\}$ as 
\begin{align} \label{partition}
    \mathcal A:= \{1\le r\le m:  \sfp(i_{r}) = \sfp(j_{r})\} , \qquad \mathcal B:= \{1,2,\dots,m\} \setminus \mathcal A.
\end{align}
For instance, as $\sfp(i_{1}) = \sfp(j_{1}) = 0$, we have $1\in \mathcal A.$ Next consider the collection
\begin{align} \label{collection}
    \Big \{{\frac{a_r -b_{\sfp(i_{r})} }{2}}: r\in \mathcal A\Big \} \cup \{ a_{r} - b_{\sfp(i_{r})} :  r\in \mathcal B  \}  \cup  \{ a_{r} - b_{\sfp(j_{r})} :  r\in \mathcal B  \} ,
\end{align}
i.e., a multiset of size $\eta:=|\mathcal A| +2|\mathcal B|  \ge m$. Specifying the above collection is motivated by the fact that the inverses/reciprocals of these terms will exactly, up to a common multiplicative factor of $2$, be the conductances of the non-curly edges that will allow the Gaussian integral \eqref{gauss} to correspond to the GFF on the associated weighted graph with a distinguished vertex which for the ease of readability will be spelt out shortly. Distinguishing the set $\cA$ allows us to work with simple graphs instead of multi-graphs. Since for any $r \in \cA,$ the term $a_{r}-b_{\sfp(i_{r})}$ appears twice in \eqref{215} on account of the fact that $a_{r}-b_{\sfp(i_{r})}=a_{r}-b_{\sfp(j_{r})}$, one can think of either two edges between $b_{\sfp(i_{r})}$ and $a_r$ or a single edge with twice the conductance. 
While all the theory of weighted graphs that we rely on work even for multi-graphs, we choose to adopt the latter  convention for notational clarity. 

However, first, it will be convenient to order the elements in \eqref{collection} 
in a non-increasing order:
\begin{align}\label{non-inc23}
    \ell^{(\mathcal I)}_1 \ge \ell^{(\mathcal I)}_2 \ge \cdots \ge \ell^{(\mathcal I)}_{\eta}.
\end{align}
Given the above, recall from Figure \ref{fig:feynmandiag2} that the graph of interest $ \wh G = (\wh V, E)$ has vertex set  $$\wh V:=\{0,a_1,b_1,\cdots,a_m,b_m\},$$ with the vertex $0$ being a distinguished vertex (the augmentation), and the edge set 
\begin{align*}
    E :=   \{ (a_r,b_r) :  1\le r \le m\} \cup   \{ (a_r,b_{\sfp(i_r)}) : r\in \mathcal A\}\cup     \{ (a_r,b_{\sfp(i_r)}) : r\in \mathcal B\} \cup \{(a_r,b_{\sfp(j_r)}):  r\in \mathcal B\}.
\end{align*} 
Note that $|\wh V| = 2m+1$. 
Finally, \(\{c_e\}_{e\in E}\) denoting the collection of the conductances is defined as twice the corresponding coefficient in the Gaussian integral  \eqref{gauss}, i.e.,
 \begin{align}
 \label{conductprof}
     \{ c_e\}_{e\in E} = \Big\{ \frac{4}{b_1-a_1},\cdots, \frac{4}{b_m-a_m},\frac{2}{\ell^{(\mathcal I)}_1},\cdots, \frac{2}{\ell^{(\mathcal I)}_{\eta}}\Big\}. 
 \end{align}
 In particular, owing to the comment following \eqref{dominance} and  \eqref{dominance2}, the first $m$ conductance values are larger than the last $\eta$ values.
 Finally, also note that every vertex in $\wh G$, {except  $0$ which has degree at most $h$}, has degree at most $3$. This is because every internal vertex $a_r$ for $r=1,2,\ldots, m,$ is connected to $b_r$ through a curly edge, and to two parent vertices $b_{\sfp(i_r)},b_{\sfp(j_r)}$ using non-curly edges. Similarly, every internal vertex $b_r$ for $r=1,2,\ldots, m-1,$ is connected to $a_r$ through a curly edge, and at most two edges connecting to the locations of the next collisions that random walk paths indexed $i_r,j_r$ are involved in. Finally, $0$ is connected to all the initial locations of the collisions of the $h$ random walks indexed $1,2, \ldots, h.$

We now arrive at the all important lower bound of the integral  in \eqref{gauss}.
\begin{lemma} \label{treebound}
Let  $ \mathcal{I}\in \textup{\textsf{Col}}^{(h,m)}$  be any collision pattern.
For any times $2\le a_1<b_1<\cdots < a_m<b_m\le 3$ satisfying \eqref{a} and \eqref{b}, the quantity \eqref{215} is lower bounded by  
\begin{align}
  (2\pi)^{2m}\cdot   2^{-3m} \cdot  3^{-2m}     \cdot    \prod_{r=1}^{m}  (b_r-a_r)\cdot   \prod_{k=\eta-m+1}^{\eta}  \ell^{(\mathcal I)}_k.
\end{align}
\end{lemma}

\begin{proof}
As indicated already, we will apply Lemma \ref{key}. Let \(\mathcal{T}=\mathcal{T}(\wh G)\) be the set of all spanning trees of \(\wh G\). We claim that
    \begin{align} \label{321}
        \sum_{T \in \mathcal{T}} \prod_{e \in E(T)} c_e \le   3^{2m}\cdot \prod_{r=1}^{m} \frac{4}{b_r-a_r} \cdot \prod_{k=\eta-m+1}^{\eta} \frac{2}{\ell^{(\mathcal I)}_k} .
       \end{align}
Let $T\in \mathcal T$ be any fixed  spanning tree.
Then $T$ has $2m$ edges (since $|\wh V|=2m+1$).  Recall from \eqref{conductprof} that the set of conductances of all the edges is \begin{align*} \{ c_e\}_{e\in E} = \Big\{ \frac{4}{b_1-a_1},\cdots, \frac{4}{b_m-a_m},\frac{2}{\ell^{(\mathcal I)}_1},\cdots, \frac{2}{\ell^{(\mathcal I)}_{\eta}}\Big\}.
 \end{align*}
 and that the first $m$ values are larger than the remaining $\eta$ values and further among the latter the last $m$ values are the largest since the $\frac{2}{\ell^{(\mathcal I)}_j}$s were arranged in increasing order. Thus
\begin{align}
    \prod_{e \in E(T)} c_e \le    \prod_{r=1}^{m} \frac{4}{b_r-a_r} \cdot \prod_{k=\eta-m+1}^{\eta}  \frac{2}{\ell^{(\mathcal I)}_k} .
\end{align}
Finally, as already noted, every vertex in $\wh G$ other than $0$ has degree at most $3.$ A standard graph theoretic result (stated and proved in the Appendix for completeness, see Lemma \ref{graph}) implies that $|\mathcal T(\wh G)| \le  3^{2m}$. Hence we obtain \eqref{321}, and thus by Lemma \ref{key} we conclude the proof. 
{Note that the Lemma \ref{key} applies since our graph $\wh G$ is connected—indeed, every vertex is connected to the distinguished vertex 0.}   
\end{proof}

Putting things together, by Lemma \ref{treebound}, for any collision pattern  $ \mathcal{I}\in \textup{\textsf{Col}}^{(h,m)}$  and times $2\le a_1<b_1<\cdots < a_m<b_m\le 3$ satisfying \eqref{a} and \eqref{b}, recalling that the pre-factor of the integral in the RHS of  \eqref{2} is 
$$\Big (  \frac{1}{\pi a_1} \Big)^2 \prod_{r=1}^{m}     G_\theta (b_r-a_r) \cdot \prod_{r=1}^{m} \frac{1}{2\pi \cdot  \frac{b_r-a_r}{4}}  \cdot \prod_{r=1}^{m-1} \Big[\frac{1}{ \pi   (  {a_{r+1} - b_{\sfp(i_{r+1})}} )  }   \cdot  \frac{1}{\pi   ({a_{r+1} - b_{\sfp(j_{r+1})}}) } \Big]$$
it follows that the entire RHS is lower bounded by 
\begin{align} \label{222}
      C ^m &\prod_{r=1}^m G_\theta (b_r-a_r) \cdot 
    \prod_{r=1}^{m} \frac{1}{b_r-a_r} \cdot  \prod_{r\in \mathcal A} \frac{1}{(a_r-b_{\sfp(i_{r})})^2 } \cdot  \prod_{r\in \mathcal B} \Big[\frac{1}{a_r-b_{\sfp(i_{r})} } \cdot \frac{1}{a_r-b_{\sfp(j_{r})} }  \Big] \nonumber \\
    &\qquad  \qquad  \qquad  \qquad  \qquad  \qquad   \cdot   \Big[ (2\pi)^{2m}\cdot 2^{-3m}  \cdot  
 3^{-2m}     \cdot \prod_{r=1}^{m}  (b_r-a_r) \prod_{k=\eta-m+1}^{\eta}  \ell^{(\mathcal I)}_k \Big]  \nonumber\\
 &\ge  C^m     \prod_{r=1}^m G_\theta (b_r-a_r) \cdot  \prod_{r\in \mathcal A} \frac{1}{((a_r-b_{\sfp(i_{r})})/2)^2 } \cdot  \prod_{r\in \mathcal B} \Big[\frac{1}{a_r-b_{\sfp(i_{r})} } \cdot \frac{1}{a_r-b_{\sfp(j_{r})} }  \Big]\cdot  \prod_{k=\eta-m+1}^{\eta}  \ell^{(\mathcal I)}_k \nonumber\\
 &=  C^m     \prod_{r=1}^m G_\theta (b_r-a_r) \cdot  \prod_{r\in \mathcal A} \frac{1}{(a_r-b_{\sfp(i_{r})})/2}  \cdot  \prod_{k=1}^{\eta-m}  \frac{1}{\ell^{(\mathcal I)}_k},
\end{align}
where in the last equality we used the fact that $\ell^{(\mathcal I)}_1 , \cdots , \ell^{(\mathcal I)}_\eta$ forms an enumeration of the elements in \eqref{collection}. Above, the constant $C$ changed from line to line but was always a constant independent of the weighted graph.

~

The task has now reduced to controlling the quantities $\ell^{(\mathcal I)}_k$ and in particular showing that they are not too large.
Recalling the definitions of $\sfp(i_r)$ and  $\sfp(j_r)$ from \eqref{221}, and that $\sfp(i_{1})=\sfp(j_{1})=0$, we aim to control the size of gaps $r-\sfp(i_r)$ along with $r-\sfp(j_r)$.  Note that given  a collision pattern $\mathcal I=(\{i_1,j_1\},...,\{i_m,j_m\}) \in \textsf{Col}^{(h,m)}$, the values of  $\sfp(i_r)$ and  $\sfp(j_r)$  are completely determined. The following lemma controls the product of gaps $r-\sfp(i_r)$ and $r-\sfp(j_r)$.  

\begin{lemma}\label{short}
Let $\mathcal I=(\{i_1,j_1\},...,\{i_m,j_m\}) \in \textup{\textsf{Col}}^{(h,m)}$ be any  collision pattern.
Recalling that  $\eta = |\mathcal A| + 2|\mathcal B|$, let 
\begin{equation}\label{noninc24}
\mathcal K^{(\mathcal I)}_1 \ge \cdots \ge \mathcal K^{(\mathcal I)}_{\eta}
\end{equation}
 be the arrangement of the elements 
\begin{align*}
    \{r - \sfp(i_{r}): r\in \mathcal A\}\cup \{r - \sfp(i_{r}): r\in \mathcal B\} \cup \{r - \sfp(j_{r}): r\in \mathcal B\} 
\end{align*} 
in a non-increasing order (note the distinction from the related labeling \eqref{non-inc23}). Then,
\begin{align}  \label{amgm}
\prod_{k=1}^{\lceil m/2 \rceil} \mathcal K^{(\mathcal I)}_k \le (2h)^{\lceil m/2 \rceil}.
\end{align}
Note that as $\eta \ge |\mathcal A| + |\mathcal B| = m \ge \lceil m/2 \rceil,$ the above product is valid.
\end{lemma}
\begin{proof}
Let $\widetilde{\mathcal K}^{(\mathcal I)}_1 \ge \cdots \ge \widetilde{\mathcal K}^{(\mathcal I)}_{2m}$ be the arrangement of the $2m$ many elements
\begin{align*}
    \{r - \sfp(i_{r}): 1\le r \le m\}\cup   \{r - \sfp(j_{r}):  1\le r \le m\} 
\end{align*} 
in a non-increasing order. As $\{\mathcal K^{(\mathcal I)}_1,\dots,\mathcal K^{(\mathcal I)}_{\eta}\}\subseteq\{\widetilde{\mathcal K}^{(\mathcal I)}_1 ,\dots, \widetilde{\mathcal K}^{(\mathcal I)}_{2m}\}$, we have 
\begin{align*}
\sum_{k=1}^{\lceil m/2 \rceil}{\mathcal K}^{(\mathcal I)}_k\le \sum_{k=1}^{\eta}{\mathcal K}^{(\mathcal I)}_k\le \sum_{k=1}^{2m}\widetilde{\mathcal K}^{(\mathcal I)}_k.
\end{align*}
{Note that for any $1\le r \le m$, if $i_r=i$, then the $i^{th}$ particle or label can be thought to jump from $\sfp(i_{r})$ to $r$ having jump length $r - \sfp(i_{r})$ (similarly for $\sfp(j_{r})$).

{Since every particle has a journey of duration at most $m$ jumping across the times $1\le r\le m$, by considering the total jump lengths of all the $h$ particles,} we obtain}
\begin{align*}  
\sum_{k=1}^{2m}\widetilde{\mathcal K}^{(\mathcal I)}_k \le mh.
\end{align*}
Thus by AM-GM inequality,
\begin{align*}  
\prod_{k=1}^{\lceil m/2 \rceil} \mathcal K^{(\mathcal I)}_k \le \Big(\frac{\sum_{k=1}^{{\lceil m/2 \rceil}}\mathcal K^{(\mathcal I)}_k }{{\lceil m/2 \rceil}}\Big)^{\lceil m/2 \rceil} \le \Big(\frac{\sum_{k=1}^{2m}\widetilde{\mathcal K}^{(\mathcal I)}_k}{{\lceil m/2 \rceil}}\Big)^{\lceil m/2 \rceil} \le (2h)^{\lceil m/2 \rceil},
\end{align*}
concluding the proof.
\end{proof}

We are now ready to conclude the proof of  Proposition \ref{heat}.
\begin{proof}[Proof of Proposition \ref{heat}]
 Recall that we have reduced the task to lower bounding \eqref{222}.   
Towards this, {assume that $m \ge 100h$, and} let  $ {\mathcal{I}\in \textsf{Col}^{(h,m)}}$ be any pattern configuration and 
\begin{equation}\label{vertexset}
2\le a_1<b_1<\cdots < a_m<b_m\le 3
\end{equation}
 be any given  times satisfying \eqref{a} and \eqref{b}.
Relating the quantities in \eqref{non-inc23} and \eqref{noninc24}, observe that 
for any $1\le r\le m$ such that  $\sfp(i_{r})  \ge 1,$
\begin{align} \label{finalfinal}
    a_{r} - b_{\sfp(i_{r})} \le      a_{r} - a_{\sfp(i_{r})} \le (r- \sfp(i_{r})) \cdot \frac{1}{m} + \frac{1}{5m} \le (r- \sfp(i_{r})) \cdot \frac{2}{m} ,
    \end{align}
and similarly for  $    a_{r} - b_{\sfp(j_{r})}$.  Note that the condition $\sfp(i_{r})  \ge 1$ is essential, since otherwise (i.e.  $\sfp(i_{r})  =0$) we have $ a_{r} - b_{\sfp(i_{r})} = a_r \ge 2$.
Also, since there are  $h$ emanating paths from the origin,
{\begin{align} \label{0bound}
    |\{ 1\le r\le m: \sfp(i_{r}) =0 \}| +    |\{ 1\le r\le m: \sfp(j_{r}) =0 \}|  \le h.
\end{align}
Here we may have a strict inequality above, since some of the emanating paths may not be involved in any collision at all.}
In order to lower bound \eqref{222}, we upper bound the product of the following elements
\begin{align}   \label{twosubset}
  \Big \{\frac{a_r-b_{\sfp(i_{r})}}{2}: r\in \mathcal A\Big\}  \cup \{ {\ell^{(\mathcal I)}_1} , \cdots \ell^{(\mathcal I)}_{\eta-m}\}
\end{align}
(recall that $\ell^{(\mathcal I)}_1  \ge \cdots \ge \ell^{(\mathcal I)}_\eta$ are arrangements of the elements in \eqref{collection}). To accomplish this,
we use the fact that every element, particularly those whose corresponding value of $\sfp(i_{r})$ or $\sfp(j_{r})$ is zero, is at most 3 (by virtue of the bounds in \eqref{vertexset}), and then apply \eqref{finalfinal}  for the product of others.
Recalling the definition of $\mathcal K^{(\mathcal I)}_1 \ge \cdots \ge \mathcal K^{(\mathcal I)}_{\eta}$ in Lemma \ref{short},  noting that $ |\mathcal A| + (\eta-m)  = m$ (since $\eta =  |\mathcal A|+2 |\mathcal B|$ and  $m=|\mathcal A|+ |\mathcal B|$),
 \begin{align*}  
   \prod_{r\in \mathcal A} \frac{a_r-b_{\sfp(i_{r})}}{2}  \cdot  \prod_{k=1}^{\eta-m}  {\ell^{(\mathcal I)}_k}  \le  3^{2h} \cdot \Big( \prod_{k=1}^{\lceil m/2 \rceil} \mathcal K^{(\mathcal I)}_k \Big)^2 \cdot   \Big(\frac{2}{m}\Big)^{m-2h} \le 3^{2h} \cdot (2h)^{m+1} \cdot   \Big(\frac{2}{m}\Big)^{m-2h} .
   \end{align*} 
Here, the appearance of $2h$ in the above exponents $2h$ and $m-2h$  (which is positive since we assumed $m \ge 100 h$) follows from  \eqref{0bound} along with the fact the element whose  $\sfp(i_{r})$ or $\sfp(j_{r})$ is zero can appear in each of  the two subsets in \eqref{twosubset}. {Further note that each $\cK^{(\mathcal I)}_k\ge 1$ and  for any $k$,  $\cK^{(\mathcal I)}_k$ can at most appear twice, once in the first product $\prod_{r\in \mathcal A} \frac{a_r-b_{\sfp(i_{r})}}{2}$ and once in $ \prod_{k=1}^{\eta-m}  {\ell^{(\mathcal I)}_k}$.  The first inequality is an immediate consequence of the above.}

This implies that the quantity \eqref{222} (and thus the quantity \eqref{2}) is lower bounded by  
\begin{align} \label{223}
      C^m  \frac{1}{h\cdot 3^{2h}m^{2h}} \Big (\frac{m}{h} \Big )^{m}  \prod_{r=1}^m G_\theta (b_r-a_r)   .
\end{align}
Next, we integrate  the above quantity with respect to $b_1<b_2<\cdots<b_m$ variables over the region \eqref{b}, given $2\le a_1<a_2<\cdots<a_m < 3$ satisfying \eqref{a}.  By Lemma \ref{lemma asym} and noting that $ (2+ \frac{2i}{2m}  - \frac{1}{5m}) -  (2+ \frac{2i-1}{2m}  + \frac{1}{10m} )= \frac{1}{5m} $, for sufficiently large $m$ (depending on $\theta$),
\begin{align} \label{boundbound}
    \int_{a_r}^{2+ \frac{2i}{2m}  - \frac{1}{5m}} G_\theta (b_r-a_r) \dd b_r  \ge   \int_{0}^{\frac{1}{5m}} G_\theta (s) \dd s \ge \frac{1}{2\log m}.
\end{align}
Thus given $2\le a_1<a_2<\cdots <a_m<3$ satisfying \eqref{a} and a  pattern configuration $\mathcal I$, the integration of \eqref{223} with respect to $b_1<b_2<\cdots<b_m$ variables over the region \eqref{b} yields the lower bound
\begin{align}\label{int234}
     C^m  \frac{1}{h\cdot 3^{2h}m^{2h}} \Big (\frac{m}{h} \Big )^{m}   \Big  ( \frac{1}{\log m} \Big )^m     .
\end{align}
Next, we integrate the above quantity over $a_1<\cdots<a_m$ variables over the region \eqref{a}. Note that, by \eqref{a}, the volume of the latter is $\left(\frac{1}{5m}\right)^m.$ Since the quantity in \eqref{int234} is independent of $a_1,\cdots,a_m$, we obtain  the lower bound 
\begin{align*}
     C^m   \Big (\frac{1}{5m}\Big  )^{m}  \frac{1}{h\cdot 3^{2h}m^{2h}} \Big (\frac{m}{h} \Big )^{m}   \Big ( \frac{1}{\log m} \Big )^m.
\end{align*}
Finally, we take a summation over all pattern configurations $\mathcal I$  over $ \textsf{Col}^{(h,m)}$, whose cardinality is at least $C^m h^{2m}$ from \eqref{col}:
\begin{align}\label{display12}
C^m  h^{2m} \cdot  \Big (\frac{1}{5m}\Big  )^{m}  \frac{1}{h\cdot 3^{2h}m^{2h}} \Big (\frac{m}{h} \Big )^{m}      \Big ( \frac{1}{\log m} \Big )^m  =C^m  \frac{1}{h\cdot 3^{2h}m^{2h}}h^m  \Big ( \frac{1}{\log m} \Big )^m  .
\end{align}
Note that $C$ in the RHS above is a universal constant.
{Taking $m = e^{Ch/2e}  \ge  100h$}, we  deduce that the above quantity is at least  
 \begin{align*}
     C^m  \frac{1}{h\cdot 3^{2h}m^{2h}} h^m \Big (\frac{2e}{Ch} \Big)^{m}  \ge   (2e)^m \frac{1}{(10m)^{2(2e/C) \log m}}   \ge e^m = \exp( \exp(Ch/2e))
 \end{align*}
 for sufficiently large $h$ which finishes the proof.

Let us comment that from this proof we learn that the dominant contribution to the $h^{th}$ moment appears from exponentially in $h$ many collisions of the random walks in, say, the Feynman diagram illustrated in Figure \ref{fig:feynmandiag}.

\end{proof}

We conclude this section with a few remarks.  
    
{\begin{remark}\label{scaling2}
In this remark, we explain, as already alluded to in Remark \ref{scaling1}, how the scaling property \eqref{scalingrelation} implies Proposition \ref{heat}  for a general terminal time $t>0$, i.e., 
\begin{align}  \label{generaltime}
 \E\Big[ \big( \mathscr{Z}_{t}^\theta(g_1) \big)^h\Big]   \ge  \exp (e^{c_1h}).
\end{align}
 The scaling relation \eqref{scalingrelation} says that 
\begin{align}  
 \mathscr{Z}^\theta_{0,t} (\dd (\sqrt{t} x),\dd (\sqrt{t} y))  \overset{\text{law}}{=} t \mathscr{Z}^{\theta + \log t}_{0,1} (\dd x,\dd y) ,
  \end{align}
i.e. the parameter \( \theta \) in the  Critical $2d$ SHF is effectively shifted to \( \theta + \log t \).  Using this, 
  \begin{align*}  
 \mathscr{Z}_{t}^\theta(g_1) = \int_{\R^2} g_1(x)   \mathscr{Z}_{t}^\theta(\dd x) &= t \int_{\R^2} g_1(x)   \mathscr{Z}^{\theta + \log t}_{1} \Big( \dd \Big(\frac{x}{\sqrt{t}}\Big)\Big)  \\
 & = t \int_{\R^2} g_1(\sqrt{t}x)   \mathscr{Z}^{\theta + \log t}_{1} (\dd x) \\
 & = \int_{\R^2} g_{1/t}(x)   \mathscr{Z}^{\theta + \log t}_{1} (\dd x) =  \mathscr{Z}_{1}^{\theta + \log t}(g_{1/t})  .
    \end{align*} 
    Consequently, this alteration of the parameter  \( \theta \)  affects solely in (i) the time variable of the heat kernel when tested against the  Critical $2d$ SHF, and (ii) the $\theta$‐dependence of $G_\theta$ in the key estimate \eqref{boundbound} used to prove Proposition \ref{heat}. For (i), this makes a shift of variables $a_1<b_1<\cdots<a_m<b_m$ by $2/t$, and the same aforementioned arguments apply. For (ii),  Lemma~\ref{lemma asym} guarantees that the bound \eqref{boundbound}  remains valid for sufficiently large \(m\) (depending on \(t\)). Consequently, we conclude that \eqref{generaltime}  continues to hold for all \( t > 0 \), provided \( h \) is taken sufficiently large depending on \( t \).
\end{remark}
}

\begin{remark}\label{erdostaylor}It was shown in \cite{multi}, as a generalization of the well known Erd\H{o}s-Taylor theorem, that if $h$ planar random walks \emph{start at the same point}, say from the origin, then their intersection local times,  i.e., the number of the times they intersect, are approximately independent standard ${\rm{Exp}}({\pi})$ variables at scale $\log n.$ Thus,  in the notation from \eqref{pairwise}, recalling that $\beta^{2}$ is taken to be $\frac{\pi}{\log n},$ the terms in $\sum_{1\le i <j \le h}\beta^{2}\cI_{i,j}$ are approximately independent which makes the lower bound in \eqref{corrlb21} essentially sharp. This is in essence the reason that underlies the sharpness of the lower bound 
in \eqref{thm:mom12} at least at the level of exponents. Things are however not completely independent and the effect of the positive correlation is expected to manifest in the  $O(1)$ multiplicative factor $C$ depending on $h$ which may be speculated to possibly grow double exponentially as well. 

However, when the paths start at different locations, which is the scenario under consideration for Proposition \ref{heat} or Theorem \ref{main}, the approximate independence does not hold any more as reflected in the lower bound being exponentially larger than that obtained from the Gaussian correlation inequality.
\end{remark}

Having proven Proposition \ref{heat}, to prove Theorem \ref{main} we will set up a comparison argument which will allow us to transfer the estimate in Proposition \ref{heat} with $g_1$ as a test function to a generic $\varphi.$ The key input is the already stated monotonicity result Proposition \ref{monotone} which we prove next. As hinted at in Section \ref{iop}, the domain Markov property of the GFF will play a crucial role in the proof.

\section{Monotonicity of correlation kernel}  \label{sec4}
To ease readability, let us recall the statement. 

\begin{proposition}\label{monotone1}
For any $(z_1,z_2,\cdots,z_h) \in (\R^2)^h$, the function $\alpha \mapsto K_t^{(h)}(\alpha z_1,\alpha z_2,\cdots,\alpha z_h)$ is non-increasing in $\alpha>0.$
\end{proposition}
\begin{proof}

Since the argument is identical for any $t>0$, we prove the result in the case $t=1$. Accordingly, we write  $K^{(h)}(\bz): = K_1^{(h)}(\bz) $.

In the expression  of the kernel $K^{(h)}(\bz)  $ in \eqref{kernel},
given a positive integer $m$, a collision pattern $ \mathcal{I}= (\{i_1,j_1\},...,\{i_m,j_m\}) \in \textsf{Col}^{(h,m)}$, and times $0\le a_1< b_1<...< a_m < b_m\le 1  $,  the integral with respect to spatial variables  $x_1,y_1,...,x_m, y_m \in  \R^2$ is written as
\begin{align} 
  &  \iint\limits_{\substack{x_1,y_1,...,x_m, y_m \in  \R^2 }}     g_{\f{a_1}{2}} (x_1-z_{i_1})  g_{\f{a_1}{2}}  (x_1-z_{j_1})   \prod_{r=1}^{m}   \Big[  G_\theta (b_r-a_r) g_{\frac{b_r-a_r}{4}} (y_r-x_r)\Big]  \nonumber \\
  &\qquad \qquad \cdot \prod_{r=1}^{m-1}  \Big[g_{\frac{a_{r+1} - b_{\sfp(i_{r+1})}}{2}}(x_{r+1}-y_{\sfp(i_{r+1})}) g_{\frac{a_{r+1} - b_{\sfp(j_{r+1})}}{2}}(x_{r+1}-y_{\sfp(j_{r+1})})\Big]\dd \bx \dd \by  \nonumber \\
   &= \Big (  \frac{1}{\pi a_1} \Big)^2 \prod_{r=1}^{m}     G_\theta (b_r-a_r) \cdot \prod_{r=1}^{m} \frac{1}{2\pi \cdot \frac{b_r-a_r}{4}}  \cdot \prod_{r=1}^{m-1} \Big[\frac{1}{\pi   (a_{r+1} - b_{\sfp(i_{r+1})})}  \cdot   \frac{1}{\pi   (a_{r+1} - b_{\sfp(j_{r+1})})} \Big]\nonumber \\
   &\qquad \cdot \iint\limits_{x_1,y_1,\cdots,x_m,y_m\in \R^2} \exp \Big ( -\frac{|x_1 - z_{i_1}|^2}{a_1} -\frac{|x_1 - z_{j_1}|^2}{a_1}   - \sum_{r=1}^m \frac{|y_r-x_r|^2 }{(b_r-a_r)/2}  \nonumber \\
   &\qquad \qquad\qquad\qquad -\sum_{r=1}^{m-1} \frac{|x_{r+1}-y_{\sfp(i_{r+1})}| ^2}{  a_{r+1} - b_{\sfp(i_{r+1})}} -  \sum_{r=1}^{m-1}\frac{|x_{r+1}-y_{\sfp(j_{r+1})}| ^2}{  a_{r+1} - b_{\sfp(j_{r+1})}}\Big)\dd \bx \dd \by  \label{integral}.
\end{align}
{Recall that (see \eqref{221} and its following discussion) if $\sfp(i_r)=0$ (resp. $\sfp(j_r)=0$), then $(b_{\sfp(i_r)}, y_{\sfp(i_r)}) =(0,z_{i_r})$ (resp. $(b_{\sfp(j_r)}, y_{\sfp(j_r)}) =(0,z_{j_r})$).}

We will in fact prove the monotonicity claim of $\alpha \mapsto K^{(h)}(\alpha \bz)  $ in $\alpha>0$ `graph-by-graph', i.e., by establishing the same for the last integral term above. The full monotonicity then follows by integrating out the graph, i.e.,   the positive integer $m$, the collision pattern $ \mathcal{I}= (\{i_1,j_1\},...,\{i_m,j_m\}) \in \textsf{Col}^{(h,m)}$, and times $0\le a_1< b_1<...< a_m < b_m\le 1  $.

The monotonicity of the integral in \eqref{integral} follows from the following more general claim by considering the graph $G$ with $h$ boundary vertices as shown in Figure \ref{fig:feynmandiag} with values specified to be $z_1,\ldots, z_h$ and the conductances given by \eqref{conductprof}.

\textbf{Claim.} Let $\alpha>0$, $h \in \mathbb N$ and $z_1,\cdots,z_h \in \R^2$. Then for any connected graph $G = (V,E)$ with a boundary $\partial V = \{w_1,\cdots,w_h\} \subseteq V$ and non-negative conductance $(c_{uv})_{\{u,v\}\in E}$, setting the boundary values  $\phi^{(\alpha)}_{w_i}:= \alpha z_i$, the integral
    \begin{align} \label{312}
     \int _{(\R^2)^{V \setminus \partial V }} \exp\Big ( - \sum_{\{u,v\}\in E} c_{uv} |\phi^{(\alpha)}_u - \phi^{(\alpha)}_v|^2\Big) \prod_{v\in V \setminus \partial V} \dd \phi_{v}^{(\alpha)}
    \end{align}
    is non-increasing in $\alpha>0.$

~

\noindent
    To verify the claim, let us examine the exponent
    \begin{align} \label{311}
        \sum_{\{u,v\}\in E} c_{uv} |\phi^{(\alpha)}_u - \phi^{(\alpha)}_v|^2.
    \end{align}
    Let  $\mathcal H=(\mathcal H_v)_{v\in V}$ be the harmonic extension of $(z_1,\cdots,z_h)$ on $\partial V$, i.e. $\mathcal H_{w_i} = z_i$ for $i=1,2,\cdots,h$ and $\Delta \mathcal H_v  = 0$ on $v\in V\setminus \partial V$, where the weighted (un-normalized) graph Laplacian $\Delta$ defined in \eqref{laplacian} acts coordinate‐wise. Note also that for any $\alpha\in \R,$ the function $\alpha\cH$ is a harmonic function on $ V\setminus \partial V$ with $\alpha \mathcal H_{w_i} = \alpha z_i$ for $i=1,2,\cdots,h.$
Consider the change of variable
    \begin{align}\label{domainmarkov}
   \widetilde \phi^{(\alpha)}_v :=        \phi^{(\alpha)}_v - \alpha \mathcal H_v,\qquad v\in V.
    \end{align}
    Then for \emph{any} $\alpha  \in \R$,  since $\phi^{(\alpha)}_{w_i}:= \alpha z_i$ and $\mathcal H_{w_i} = z_i$ for $i=1,\cdots,h$, we have  $\widetilde \phi^{(\alpha)}_v = (0,0)$ on the  boundary $\partial V$. Further, the quantity \eqref{311} can be written as
        \begin{align*}
        &\sum_{\{u,v\}\in E} c_{uv} |\widetilde \phi^{(\alpha)}_u + \alpha\mathcal H_u - \widetilde \phi^{(\alpha)}_v-\alpha\mathcal H_v|^2 \\
        &=   \sum_{\{u,v\}\in E} c_{uv} |\widetilde \phi^{(\alpha)}_u  - \widetilde \phi^{(\alpha)}_v|^2 +\alpha^2   \sum_{\{u,v\}\in E} c_{uv} | \mathcal H_u -\mathcal H_v|^2 - 2\alpha  \sum_{\{u,v\}\in E}c_{uv}(\widetilde \phi^{(\alpha)}_u  - \widetilde \phi^{(\alpha)}_v) \cdot (\mathcal H_u -\mathcal H_v)   .
    \end{align*}
    Integrating by parts (see Lemma \ref{ibp} in Appendix), the last summation term above is written as
    \begin{align*}
        \sum_{\{u,v\}\in E}c_{uv}(\widetilde \phi^{(\alpha)}_u  - \widetilde \phi^{(\alpha)}_v) \cdot (\mathcal H_u -\mathcal H_v)   =\sum_{v\in  V} \widetilde{\phi}^{(\alpha)}_v \cdot  \Delta \mathcal{H}_v= \sum_{v\in \partial V} \widetilde{\phi}^{(\alpha)}_v  \cdot  \Delta \mathcal{H}_v =0.
    \end{align*}
Above, the second last equality holds since, by hypothesis,   $(\mathcal H_v)_{v\in V}$ is harmonic in $V\setminus \partial V$,
  and the last identity holds since $\widetilde \phi^{(\alpha)}_v=(0,0)$ on $\partial V$. Hence, the integral \eqref{312} is equal to
    \begin{align*}
            &\int _{(\R^2)^{V \setminus \partial V }}  \exp\Big ( - \sum_{\{u,v\}\in E} c_{uv} |\widetilde \phi^{(\alpha)}_u  - \widetilde \phi^{(\alpha)}_v|^2 -\alpha^2   \sum_{\{u,v\}\in E} c_{uv} | \mathcal H_u -\mathcal H_v|^2 \Big) \prod_{v\in V \setminus \partial V} d \widetilde\phi_{v}^{(\alpha)}\\
              &=\exp\Big ( -\alpha^2   \sum_{\{u,v\}\in E} c_{uv} | \mathcal H_u -\mathcal H_v|^2 \Big)  
\int _{(\R^2)^{V \setminus \partial V }} \exp\Big ( - \sum_{\{u,v\}\in E} c_{uv} |\widetilde \phi^{(\alpha)}_u  - \widetilde \phi^{(\alpha)}_v|^2 \Big) \prod_{v\in V \setminus \partial V} d \widetilde\phi_{v}^{(\alpha)}.
    \end{align*}
    Since  $\widetilde\phi_{v}^{(\alpha)}$ is an integration variable for  $v\in V \setminus \partial V$ and $\widetilde \phi^{(\alpha)}_v=(0,0)$ on $\partial V$ for any $\alpha$, the last integral term above is independent of $\alpha$. The pre-factor $\exp\big ( -\alpha^2   \sum_{\{u,v\}\in E} c_{uv} | \mathcal H_u -\mathcal H_v|^2 \big)$ is clearly monotonically decreasing in $\alpha$ which concludes the proof.
\end{proof}

\begin{remark}While this is a well known fact, nonetheless, let us remark that the decomposition in \eqref{domainmarkov} for $\alpha=1$ is indeed how one proves the domain Markov property of the GFF. Finally, while the flat data case has been the focus of this article leading to the one sided kernel in \eqref{kernel}, one may also define a two sided version where in addition to the initial points $z_1,\ldots, z_h,$  the final points  $w_1, \ldots, w_h$ are not integrated over but are fixed. The same argument as above implies a similar monotonicity statement in this case as well. 
\end{remark}

\section{From Gaussian to compactly supported test functions}  \label{sec5}
Having proved Proposition \ref{heat} and the monotonicity result Proposition \ref{monotone}, the only missing piece in the proof of Theorem \ref{main} is the following lemma, which compares the integrals of two test functions—namely, the Gaussian test function and a compactly supported function—against the kernel \(K^{(h)}(\bz)\).

\begin{lemma} \label{compare}
 Let  $0<\delta<1$.  Recalling that we set $K^{(h)}(\bz) = K^{(h)}_1(\bz)$, for sufficiently large $h$,
\begin{align*}
 \int_{([-\delta,\delta]^2)^h} K^{(h)}(\bz)  \dd \bz \ge   \delta^{2h}e^{-2h^2} \int_{(\mathbb{R}^2)^h} K^{(h)}(\bz)g_1^{\otimes h} (\bz)  \dd \bz .
\end{align*}
 
\end{lemma}

Before proving the lemma, let us first finish the proof of Theorem \ref{main} which is an immediate consequence of this lemma and Proposition \ref{heat}.

\begin{proof}[Proof of Theorem \ref{main}]
Since $\varphi$ is a smooth non-negative function such that $\varphi(0)>0$, there exist constants  $c >0$ and $0<\delta<1$ such that
  \begin{align}
      \varphi (z)\ge c \1_{[-\delta,\delta]^2} (z) ,\qquad \forall z \in \R^2.
  \end{align}
Using \eqref{hmoment}, Lemma \ref{compare} and Proposition \ref{heat}, for sufficiently large $h$,
  \begin{align*}
      \E\Big[ \big( \mathscr{Z}_{1}^\theta(\varphi) \big)^h\Big]   \ge \frac{c^h}{2^h}   \int_{([-\delta,\delta]^2)^h} K^{(h)}(\bz)  \dd \bz & \ge    \frac{c^h}{2^h}   \delta^{2h}e^{-2h^2} \int_{(\mathbb{R}^2)^h} K^{(h)}(\bz)g_1^{\otimes h} (\bz)  \dd \bz  \\
      &= c^h\delta^{2h}e^{-2h^2} \E\Big[ \big( \mathscr{Z}_{1}^\theta(g_1) \big)^h\Big]   \\
      &\ge   c^h\delta^{2h}e^{-2h^2}  \cdot  \exp (e^{c_1h})  \ge \exp (e^{c_1h/2}),
  \end{align*} which concludes the proof. 
\end{proof}

We conclude this section by providing the proof of Lemma \ref{compare}.
 
 \begin{proof}[Proof of Lemma \ref{compare}]
It suffices to prove that  for large enough $h$,
    \begin{align} \label{343}
\int_{(\mathbb{R}^2)^h} K^{(h)}(\bz)g_1^{\otimes h} (\bz)  \dd \bz \le   {\delta^{-2h}}  e^{2h^2} \int_{([-\delta,\delta]^2)^h} K^{(h)}(\bz)g_1^{\otimes h} (\bz)  \dd \bz.
\end{align}
Indeed, given this, as $g_1^{\otimes h} (\bz) \le 1 $ for all $\bz\in (\mathbb{R}^2)^h$ (since $g_1\le 1$ pointwise by \eqref{heatkernel1}),
\begin{align*}
 \int_{([-\delta,\delta]^2)^h} K^{(h)}(\bz)  \dd \bz \ge     \int_{([-\delta,\delta]^2)^h} K^{(h)}(\bz)  g_1^{\otimes h} (\bz) \dd \bz  \overset{\eqref{343}}{\ge}  
  \delta^{2h} e^{-2h^2} \int_{(\mathbb{R}^2)^h} K^{(h)}(\bz)g_1^{\otimes h} (\bz)  \dd \bz ,
\end{align*}
concluding the proof of lemma.\\

\noindent
To prove \eqref{343}, set $Q_\delta:=([-\delta,\delta]^2)^h$.
Then we write
\begin{align} \label{330}
\int_{(\mathbb{R}^2)^h}K^{(h)}(\bz)g_1^{\otimes h} (\bz)\dd \bz 
&= {\int_{Q_\delta}K^{(h)}(\bz)g_1^{\otimes h} (\bz)\dd \bz}  +  {\int_{Q_\delta^c}K^{(h)}(\bz)g_1^{\otimes h} (\bz)\dd \bz}
=: I_{Q_\delta}+I_{Q_\delta^c}.
\end{align}
We aim to upper bound \( I_{Q_\delta^c} \) in terms of \( I_{Q_\delta} \). Define 
% \red{pick infinity or 2-norm?}
\begin{align} \label{342}
f(\bz) :=\|\bz\|_\infty = \max_{1\le i\le h} |z_i|_\infty,\qquad \bz = (z_1,\dots,z_h)\in (\mathbb{R}^2)^h,
\end{align}
where  $|z_i|_\infty:= \max \{|z_{i,1}|,|z_{i,2}|\} $ for  $z_i = (z_{i,1},z_{i,2})\in \R^2.$
Then  \(f\) is \(1\)-Lipschitz {with respect to the $\ell^2$-norm} and is differentiable almost everywhere with  $|\nabla f(\bz)| = 1$ (recall that $|\cdot| := |\cdot|_2$ denotes the $\ell^2$-norm).

To control the integral in \eqref{330}, {
we use the following co-area formula \cite[Theorem 3.1]{coarea}}:  For any Lipschitz function $f : \R^d \to \mathbb{R}$ and a 
  nonnegative function  $u : \R^d \to [0,\infty),$ 
\begin{align*}
\int_{\R^d} u(\bz)  |\nabla f(\bz)|  \dd \bz = \int_{0}^{\infty} \Big( \int_{f^{-1}(t)} u(\bz)  \dd \mathcal{H}^{d-1}(\bz) \Big) \dd t,
\end{align*}
where   \( \mathcal{H}^{d-1} \) denotes the \((d-1)\)-dimensional Hausdorff measure.
Applying this for our $f$ defined  in \eqref{342} on $(\mathbb{R}^2)^h$, we deduce that for any nonnegative function  $u : (\mathbb{R}^2)^h \to [0,\infty),$ 
\begin{align*}
\int_{(\mathbb{R}^2)^h} u(\bz)  \dd \bz = \int_{0}^{\infty} \Bigl( \int_{f^{-1}(r)} u(\bz) \dd \mathcal{H}^{2h-1}(\bz) \Bigr)  \dd r.
\end{align*}
By a change of variables ($\ell^\infty$-spherical coordinates), writing $r= f(\bz)$ and $\bw = \frac{\bz}{f(\bz)},$
\begin{align*}
\dd \mathcal{H}^{2h-1}(\bz) = r^{2h-1} \dd\sigma(\bw),
\end{align*}
where \(\dd\sigma(\bw)\) denotes the surface measure on \(S  := \{ \bw\in (\mathbb{R}^2)^h : \|\bw\|_\infty = 1 \} \). 
In view of this, 
\begin{align*}
I_{Q_\delta^c} = \int_{\delta}^\infty \int_{ S} K^{(h)}(r\bw)  g_1^{\otimes h} (r\bw)  r^{2h-1}\dd\sigma(\bw) \dd r.
\end{align*}
The outside integral ranges from $\delta$ to $\infty$ since $Q_\delta^c$ is indeed the set of all points where $f>\delta.$
Now note that for  $r \ge \delta$ and $\bw\in S$, by the monotonicity of the kernel (Proposition \ref{monotone}),  
\begin{align*}
K^{(h)}(r\bw)g_1^{\otimes h} (r\bw) r^{2h-1}\le K^{(h)}(\delta\bw)g_1^{\otimes h} (r\bw) r^{2h-1} \le K^{(h)} (\delta \bw)   \frac{1}{(2\pi)^{h}}\exp\Bigl(-\frac{r^2}{2}\Bigr) r^{2h-1}.
\end{align*}
{Here we used the fact that since $|\bw| =|\bw| _2 \ge 1$ for any $\bw\in S,$ $$g_1^{\otimes h} (r\bw)\le \frac{1}{(2\pi)^h}\exp\Bigl(-\frac{r^2}{2}\Bigr).$$}
Thus
\begin{align} \label{331}
I_{Q_\delta^c} \le \frac{1}{(2\pi)^{h}} \Big( \int_{ S} K^{(h)} (\delta \bw)   \dd\sigma(\bw) \Big)
\Big( \int_{\delta}^\infty r^{2h-1}\exp\Bigl(-\frac{r^2}{2}\Bigr) \dd r \Big). 
\end{align}
Using a change of variable $u:=\tfrac{r^2}{2},$ for large enough $h$, 
\begin{align*}
\int_{\delta}^\infty r^{2h-1}\exp\Bigl(-\frac{r^2}{2}\Bigr) \mathrm{d}r
&=\int_{\delta^2/2}^\infty (2u)^{h-\frac{1}{2}}e^{-u} \frac{\mathrm{d}u}{\sqrt{2u}}
\\
&= {2^{ h-1}}\int_{\delta^2/2}^\infty u^{ h-1}e^{-u} \mathrm{d}u\le {2^{ h-1}} \Gamma(h)
\le e^{h\log h},
\end{align*}
where $\Gamma(h)=\int_{0}^\infty u^{ h-1}e^{-u} \mathrm{d}u=(h-1)!$ denotes the Gamma function. The above bound is a simple consequence of Stirling's formula and that $e>2.$ Hence, applying this to \eqref{331}, for sufficiently large $h$,
\begin{align} \label{332}
I_{Q_\delta^c} \le e^{h\log h}\frac{1}{(2\pi)^{h}} \Big( \int_{ S} K^{(h)} (\delta \bw)   \dd\sigma(\bw) \Big)
.
\end{align}
{Similarly, 
\begin{align*}
I_{Q_\delta} = \int_{0}^{\delta} \int_{S} K^{(h)}(r\bw)g_1^{\otimes h} (r\bw) r^{2h-1}\dd\sigma(\bw) \dd r.
\end{align*} 
For \(r\in [0,\delta]\) and \(\bw\in S\), again by the monotonicity of the kernel,
\begin{align*}
K^{(h)}(r\bw)g_1^{\otimes h} (r\bw) \ge K^{(h)} (\delta \bw)   \frac{1}{(2\pi)^{h}} \exp\Bigl(- \frac{2h\delta^2}{2}\Bigr).
\end{align*}
Here we used the fact that $|\bw| =|\bw| _2 \le \sqrt{2h}$ for any $\bw\in S.$}
Thus
\begin{align*}
I_{Q_\delta} &\ge e^{-h\delta^2}  \frac{1}{(2\pi)^{h}} \left( \int_{0}^{\delta} r^{2h-1} \dd r \right) \left( \int_{S} K^{(h)} (\delta \bw)  \dd\sigma(\bw) \right) \\
&= e^{-h\delta^2}\frac{1}{(2\pi)^{h}} \cdot \frac{\delta^{2h}}{2h}  \left( \int_{S} K^{(h)} (\delta \bw)  \dd\sigma(\bw) \right).
\end{align*}
Combining this with \eqref{332}, as $0<\delta<1$, for sufficiently large $h$,
\begin{align*}
I_{Q_\delta^c} &\le e^{h\log h}  e^{h\delta^2} \frac{2h}{\delta^{2h}}  I_{Q_\delta}  \le  {\delta^{-2h}} e^{h^2} I_{Q_\delta},\\
&\implies I_{Q_\delta^c}+I_{Q_\delta}\le {\delta^{-2h}} e^{2h^2} I_{Q_\delta}.
\end{align*}
The proof of \eqref{343} now follows from \eqref{330}.
  \end{proof}

\begin{remark}\label{upperarg}Instead of having a detailed discussion on what might be involved in improving the upper bound of $\exp(\exp (h^2))$ to match the lower bound we simply mention that much of the difficulty is involved in estimating integrals of the form 
\begin{equation}\label{integralest}
\int_{(x_1,\ldots, x_{m+h-1})\in [0,1]^{m+h-1}}\prod_{i=1}^{m}\frac{1}{(x_{i}+x_{i+1}+\ldots +x_{i+h-1})}\dd x_1\dd x_2\ldots \dd x_{m+h-1}.
\end{equation}
When $h=2$, one can employ an inductive argument (see \cite{cosco, shrinking}) to argue that the above integral {grows at most exponentially in $m$. Indeed, this is the correct order of growth as a lower bound can be seen via the following reasoning. Setting $X_1,X_2,\cdots,X_{m+1}$ to be i.i.d. uniform random variables on $[0,1]$, the above integral (with $h=2$) is written as
\begin{align} \label{jensen}
\mathbb E \Big[\frac{1}{(X_1+X_2) \cdots (X_m+X_{m+1})}\Big]  .
\end{align}
Taking the logarithm, by Jensen's inequality,
\[
\log \mathbb E \Big[\frac{1}{(X_1+X_2) \cdots (X_m+X_{m+1})}\Big] \ge \mathbb E \log  \Big[\frac{1}{(X_1+X_2) \cdots (X_m+X_{m+1})}\Big] = -m \mathbb E  \log (X_1+X_2).
\]
Observe that by   Jensen's inequality again,
\[
\mathbb E  \log (X_1+X_2) <\log  \mathbb E  (X_1+X_2) = \log 1 = 0
\]
(note that we have a strict inequality since $X_1+X_2$ is a non-degenerate random variable).
This shows that the quantity \eqref{jensen} is at least $e^{cm}$ for some constant $c>0$.}

 However, for $h\ge 3$, even showing finiteness of this integral is non-trivial. However, an inductive argument as above can be used to prove an exponential in $m$ upper bound. Nonetheless, it is reasonable to expect that the dominant contribution comes from when all the $x_{i}$s are $O(1)$ leading to an estimate of $\frac{1}{h^m}$ (up to an exponential in $m$ term). This will be taken up in future work. 
\end{remark}

\section{Upper tail estimates}  \label{sec6}
Given Theorem \ref{main}, we are now in a position to prove  Theorem \ref{tail}.
Let us start by recalling the statement which states that if $\varphi$  is a compactly supported positive smooth function on $\R^2$ such that $\varphi(0)>0$, then, for any $z>0,$
    \begin{align}\label{tail234}
\exp\Bigl(-   (\log z)^{(\log \log z)^{1+o(1)}}\Bigr) \le   \P(X_\vp  > z)\le  \exp\Bigl(-\Omega(1) \cdot \log z \cdot \sqrt{\log\log z}\Bigr).
    \end{align}
    Here, $o(1) \rightarrow 0$ as $z\rightarrow \infty$  and $\Omega(1)$ remains bounded below by a positive constant.

\begin{proof}[Proof of Theorem \ref{tail}]
Throughout the proof, we simplify the notation as {$X:=X_\varphi = \mathscr{Z}_{1}^\theta(\varphi) \ge 0$.}
Recall from \eqref{upperbound1} that the upper bound for the moment of the Critical $2d$ SHF is
\begin{align}
\E[X^h] &\le \exp\bigl(\exp({ch^2})\bigr),
\end{align}
for some $c>0.$
By Markov’s inequality, for any $t>0,$
\begin{align*}
\P(X \ge t) \le \frac{\E[X^h]}{t^h} \le   \exp\bigl(e^{c h^2} - h\log t\bigr).
\end{align*}
Taking
\begin{align*}
h &=  \Big\lfloor \frac{1}{\sqrt{c}}\sqrt{\log\log t} \Big \rfloor,
\end{align*} 
we obtain
\begin{align} \label{upper}
\P(X \ge t) \le t \cdot  \exp\bigl(-\frac{\log t}{\sqrt c} \sqrt{\log\log t} \bigr)= \exp\bigl(-\frac{(1+o(1))}{\sqrt{c}}\log t \sqrt{\log\log t} \bigr).
\end{align}
Relying on this a priori upper tail upper bound, we now see how the moment lower bound established in Theorem \ref{main} yields upper tail lower bounds.  
For $z>100$ (a large number), we introduce the following notations:
\begin{align*}
L &:= \log\log z, & M &:= \log L =  \log \log\log z ,\\
h &:= L M-10, & w &:= \exp\bigl(\exp(c L^2M^2)\bigr).
\end{align*}
Then we write
\begin{align*}
\E[X^h] &= \int_0^\infty h t^{h-1}\P(X \ge t) \dd t \\ &= \int_0^z h t^{h-1}\P(X \ge t) \dd t + \int_z^w h t^{h-1}\P(X \ge t) \dd t + \int_w^\infty h t^{h-1}\P(X \ge t) \dd t := I_1 + I_2 + I_3.
\end{align*}
On $[0,z]$,
\begin{align} \label{411}
I_1 &= \int_0^z h t^{h-1}\P(X \ge t) \dd t \le \int_0^z h t^{h-1}\dd t = z^h.
\end{align}
On $[z,w]$,
\begin{align} \label{412}
I_2 &= \int_z^w h t^{h-1}\P(X \ge t) \dd t \le \int_z^w  h t^{h-1}\P(X \ge z)\dd t \le   w^h\P(X \ge z).
\end{align}
On $[w,\infty)$, noting that $\log\log t \ge \log\log w  = c L^2M^2$ for $t \ge w$,
\begin{align*}
\P(X \ge t) \overset{\eqref{upper}}{\le}  t \cdot  \exp\bigl(-\frac{\log t}{\sqrt c} \sqrt{\log\log t} \bigr) 
\le t^{-LM+1},
\end{align*}
implying that  for sufficiently large $z$,
\begin{align} \label{413}
I_3 &= \int_w^\infty h t^{h-1}\P(X \ge t) \dd t \le h\int_w^\infty t^{h-1-LM+1}\dd t = C h w^{-9} \le 1.
\end{align}
Hence combining \eqref{411}-\eqref{413},
\begin{align} \label{415}
\E[X^h] \le z^h + w^h\P(X \ge z) + 1.
\end{align}
Note that by Theorem \ref{main}, for sufficiently large $z$,
\begin{align*}
\log \E[X^h] \ge e^{c_0h}    \ge 2 (LM -10)e^L= 2\log(z^h),
\end{align*}
implying that  $z^h \le \E[X^h]/3.$ Thus, applying this to    \eqref{415},
\begin{align*}
\P(X \ge z) \ge \frac{1}{w^h}\cdot \frac{\E[X^h]}{2}   \ge 
\exp(- h\log w)  &\ge \exp (-LM e^{c L^2M^2}  ) = \exp\Bigl(-   (\log z)^{(\log \log z)^{1+o(1)}}\Bigr),
\end{align*}
which finishes the proof.

\end{proof}

\begin{remark}\label{elab}The above proof allows us to elaborate on Remark \ref{optrem}. Note that if the upper and lower bounds of $\E[X^h]$ matched to be $\exp(\exp ({c}h))$ {for some constant $c = c(\theta)>0$} (we will ignore lower order factors in this discussion), first of all, using the upper bound, the upper tail upper bound in \eqref{upper} would improve to $\exp\bigl(-\Omega(1) \cdot {\log t}\cdot \log\log t \bigr).$ Feeding that into the lower bound argument results in essentially only one change where we define $w= \exp\bigl(\exp({c}LM)\bigr)$ which leads to a final bound of the form \begin{align*}
\P(X \ge z) \ge \frac{1}{w^h}\cdot \frac{\E[X^h]}{2}   \ge 
\exp(- h\log w)  &\ge \exp (-LM e^{{c}LM}  )  \ge      \exp\Bigl(-   (\log z)^{{O(1)} \cdot \log \log \log z}\Bigr). 
\end{align*}
\end{remark}

\section{Appendix}

In this appendix, we provide the proofs of some of the statements used in the main body of the paper.
The first is a crude upper bound on the number of spanning trees of a connected graph, in terms of degree of vertices.
\begin{lemma} \label{graph}
Let \(G\) be a connected graph on \(n\) vertices with vertex degrees \(d_1, d_2, \dots, d_n\). Then the number \(\tau(G)\) of spanning trees of \(G\) satisfies that, for any $1\le k\le n,$
\[
\tau(G) \le  \frac{1}{d_k}\cdot  \prod_{i=1}^{n} d_i.
\]
\end{lemma}

\begin{proof}
Let $d(v)$ be a degree of a vertex $v$.
Choose an arbitrary vertex \(o\) to serve as the root. In every spanning tree \(T\) of \(G\), each vertex \(v \neq o\) is connected to the tree by a unique edge (called \emph{parent edge}), which connects \(v\) to a vertex that lies on the unique path from \(v\) to the root \(o\). There are at most \(d(v)\) choices for the edge that connects \(v\) to its parent.

Thus, an \emph{encoding} of the spanning tree can be obtained by specifying, for each vertex \(v \neq o\), which of the \(d(v)\) edges incident on \(v\) is used to connect \(v\) to its parent. This encoding is injective, meaning that no two distinct spanning trees yield the same collection of choices. This is because the edge set of a spanning tree is completely determined by the set of $n-1$ encoded edges. Consequently, the number of spanning trees is at most
\[
\tau(G) \le \prod_{v \neq o} d(v).
\]
Since the root is arbitrary, the result follows.
\end{proof}

The final result included in this appendix records a well known integration by parts result for graphs. As required by our application, our graph $G = (V,E,c)$ will be weighted with $c={(c_e)}_{e\in E}$ being the conductance vector. 

\begin{lemma}\label{ibp}
Let $d \in \mathbb N$ and $G = (V,E,c)$ be any finite weighted graph. For any functions $f,g:V \rightarrow \R^d,$
    \begin{align*}
\sum_{\{x,y\}\in E}c_{\{x,y\}}(f(x)-f(y))\cdot (g(x)-g(y))=\sum_{u\in V}f(u)\cdot  \Delta g(u).
\end{align*}
Here, $\Delta$ denotes the weighted graph Laplacian matrix defined in \eqref{laplacian} which acts  coordinate‐wise in $\R^d$.
\end{lemma}

\begin{proof}
Since the Laplacian acts coordinate-wise, it suffices to consider $d=1.$
  Let $\vec E$ be the collection of directed edges in $G$ where every edge $\{u,v\}$ appears with two orientations, one pointing to $u$ denoted by $(u,v)$ and the other to $v$ denoted by $(v,u)$. Thus $|\vec E| = 2|E|.$ Using this note that 
\begin{align*}
\sum_{\{x,y\}\in E}c_{\{x,y\}}(f(x)-f(y))(g(x)-g(y))
&=\frac{1}{2}\sum_{(u,v)\in\vec E}c_{\{u,v\}}(f(u)-f(v))(g(u)-g(v))\\
&=\frac{1}{2}\sum_{(u,v)\in\vec E}\Bigl[f(u)c_{\{u,v\}}(g(u)-g(v))-f(v)c_{\{u,v\}}(g(u)-g(v))\Bigr]\\
\text{interchanging}\, u \, \text{and}\, v\, \text{in the second term}&=\frac{1}{2}\Bigl(\sum_{(u,v)\in\vec E}f(u)c_{\{u,v\}}(g(u)-g(v))
    -\sum_{(v,u)\in\vec E}f(u)c_{\{u,v\}}(g(v)-g(u))\Bigr)\\
    &=\frac{1}{2}\Bigl(\sum_{(u,v)\in\vec E}f(u)c_{\{u,v\}}(g(u)-g(v))
    +\sum_{(u,v)\in\vec E}f(u)c_{\{u,v\}}(g(u)-g(v))\Bigr)\\
&=\sum_{(u,v)\in\vec E}f(u)c_{\{u,v\}}(g(u)-g(v))\\
&=\sum_{u\in V}f(u)\,\underbrace{\sum_{v\sim u}c_{\{u,v\}}(g(u)-g(v))}_{\Delta g(u)}=\sum_{u\in V}f(u)\,\Delta g(u).
\end{align*}

\end{proof}

\bibliographystyle{plain}
\bibliography{SHF}

\end{document}